\documentclass{article}
\usepackage{amsmath, amsthm, amssymb}
\usepackage[dvips]{graphicx}
\usepackage{enumerate}
\def\authorsPS{}
\newcommand\authPS[8]{\ifnum\authPSc<1\def\authPSc{1}\else, \fi{\large\sffamily #1 #2}%
\edef\authorsPS{\authorsPS \par\vskip 1mm\noindent #1 #2:\hskip 5mm #3, #4, #5, #6, #7, #8}}

\newtheorem{theorem}{Theorem}[section]
\newtheorem{lemma}[theorem]{Lemma}
\newtheorem{corollary}[theorem]{Corollary}
\newtheorem{proposition}[theorem]{Proposition}
\theoremstyle{definition}
\newtheorem{definition}[theorem]{Definition}
\theoremstyle{remark}
\newtheorem{remark}[theorem]{Remark}

\def\authPSc{0}
\newcommand{\beq}[1]{
\begin{equation}\label{#1}}
\newcommand{\eeq}{\end{equation}}
\newcommand{\req}[1]{{\rm(\ref{#1})}}

\newcommand{\hten}{{\mathfrak H}}
\newcommand{\law}{\stackrel{\cal L}{\longrightarrow}}
\newcommand{\Nt}{\lfloor nt \rfloor}
\newcommand{\Xhat}{\hat{X}_{\frac jn}}

\begin{document}


\title{Weak convergence of the Stratonovich integral with respect to a class of Gaussian processes}

\author{Daniel Harnett, David Nualart\thanks{
D. Nualart is supported by the NSF grant DMS0904538. \newline
  }   \\
  Department of Mathematics\thinspace ,\ University of Kansas\\
405 Snow Hall\thinspace ,\ Lawrence, Kansas 66045-2142 }
\maketitle

\begin{abstract}
For a Gaussian process $X$ and smooth function $f$, we consider a Stratonovich integral of $f(X)$, defined as the weak limit, if it exists, of a sequence of Riemann sums.  We give covariance conditions on $X$ such that the sequence converges in law.  This gives a change-of-variable formula in law with a correction term which is an It\^o integral of $f'''$ with respect to a Gaussian martingale independent of $X$. The proof uses Malliavin calculus and a central limit theorem from \cite{NoNu}. This formula was known for fBm with $H=1/6$  \cite{NoRevSwan}.  We extend this to a larger class of Gaussian processes.      
\end{abstract}

\section{Introduction}
Let $X = \{X_t, t\ge 0\}$ be a centered Gaussian process, and let $f:{\mathbb R}\to{\mathbb R}$ be a ${\cal C}^\infty$ function such that $f$ and its derivatives have at most polynomial growth.  We define the Stratonovich integral of $f'$ with respect to $X$, denoted,
$$\int_0^t f'(X_s) d^\circ X_s,$$
as the limit in probability of the trapezoidal Riemann sum,
\beq{Rsum}\frac 12 \sum_{i=0}^{\Nt -1} \left[f'(X_{\frac{i+1}{n}}) + f'(X_{\frac{i}{n}})\right]\left( X_{\frac{i+1}{n}} - X_{\frac{i}{n}}\right)\eeq
when that limit exists.
 In a 2005 paper, Gradinaru {\em et al.} (\cite{GNRV}) studied this integral, and identified conditions on $X$ under which the Riemann sum converges.  In particular, the sum converges for any fractional Brownian motion (fBm) $\{B_t, t\ge 0\}$ with Hurst parameter $H > 1/6$, in which case the following change-of-variable formula holds:
\beq{gnrv} f(B_t) = f(0) + \int_0^t f'(B_s) d^\circ B_s.\eeq
Subsequently, \cite{NoRevSwan} examined the end point case $H=1/6$.  Here, it was proved that \req{Rsum} converges weakly to an It\^o-like expansion formula, consisting of a stochastic integral and a correction term in the form of an It\^o integral of $f'''$.  In this case, we have the weak change-of-variable formula,
\beq{nrs} f(B_t) \stackrel{\cal L}{=} f(0) + \int_0^t f'(B_s) d^\circ B_s - \frac{\sqrt{6}}{12}\int_0^t f^{(3)}(X_s)dW_s,\eeq
where $W$ is a Brownian motion, independent of $B$, with variance given by \req{eta_def}.

In this paper, we consider the behavior of \req{Rsum} for a more general class of Gaussian processes, which are characterized by conditions on the covariance.  Convergence follows from a central limit theorem first proved by Nourdin and Nualart in \cite{NoNu}.  This theorem is based on Malliavin calculus, and applies to a sequence of multiple Skorohod integrals.  In this paper, we give a set of six covariance conditions on the process $X$, which lead to weak convergence of the form \req{nrs}.  These conditions are satisified by fBm with $H = 1/6$.  As an application, we have found that the conditions are met for three fBm-derived processes, including
\begin{itemize}
\item  Bifractional Brownian motion (bBm) with parameters $HK = 1/6$; 
\item  `extended' bBm with $K \in (1,2)$ and $HK = 1/6$; and
\item  sub-fractional Brownian motion with parameter $h = 1/3$.
\end{itemize}

In the prequel to this paper (\cite{HaNu}), we applied the same central limit theorem to a `midpoint' Riemann sum of the form
\[\sum_{j=1}^{\left\lfloor \frac{nt}{2}\right\rfloor} f'(X_{\frac{2j-1}{n}}) (X_{\frac{2j}{n}} - X_{\frac{2j-2}{n}} ).\]
For this sum, we found a slightly different weak change-of-variable formula for a process that acts similar to fBm with $H=1/4$.  In that case, the sum converges weakly to a stochastic integral plus a correction term in the form of an It\^o integral of $f''$, namely
\[f(X_t) \stackrel{\cal L}{=} f(X_0) + \int_0^t f'(X_s)^\circ dX_s +\frac 12\int_0^t f''(X_s)dB_s,\]
where $B$ is a scaled Brownian motion, independent of $X$, with a given variance.  As is suggested in \cite{GNRV}, there are other forms of Riemann sums that can be tried, including one for any fBm with $H > 1/10$.  We expect that our theoretical tools could be applied to the $H=1/10$ case as well, but this is not pursued in the present paper.

A brief outline of the paper is as follows.  In Section 2, we give some background on Malliavin calculus, and the main analytical tools that will be used in this paper.  We also recall the central limit theorem (proved in \cite{HaNu}) that provides the main theoretical basis for our result.  In Section 3, we identify the covariance conditions on the process $X$ for the CLT to hold, and prove \req{nrs}.  The proof essentially consists of restating the Riemann sum \req{Rsum} as a sequence of terms dominated in probability by 3-fold Skorohod integrals; then verifying that the CLT conditions are met.  Section 4 discusses the three examples listed above, and demonstrates the procedure for verifying the covariance conditions.  In Section 5 we give proofs for three of the longer lemmas from Section 3.

The main inspirations for this paper were \cite{NoNu} and \cite{NoRevSwan}.  Most of the notation follows that in \cite{NoNu} and \cite{HaNu}.

\section{Preliminaries and notation}
Let $X = \{X(t) , t \ge 0\}$ be a centered Gaussian process defined on a probability space $( \Omega, {\cal F}, P )$ with continuous covariance function 
$${\mathbb E}[X(t) X(s) ] = R(t,s),\;\; s,t \ge 0.$$
We will always assume that ${\cal F}$ is the $\sigma-$algebra generated by $X$.  Let $\cal E$ denote the set of step functions on $[0, T]$ for $T >0$; and let $\hten$ be the Hilbert space defined as the closure of $\cal E$ with respect to the scalar product
$$\left< {\mathbf 1}_{[0,t]} , {\mathbf 1}_{[0,s]} \right>_\hten = R(t,s),\;\; s,t \ge 0.$$  The mapping ${\mathbf 1}_{[0,t]} \mapsto X(t)$ can be extended to a linear isometry between $\hten$ and the Gaussian space spanned by $W$.  We denote this isometry by $h \mapsto X(h)$.  In this way, $\{ X(h), h \in \hten\}$ is an isonormal Gaussian process.  For integers $q \ge 1$, let $\hten^{\otimes q}$ denote the $q^{th}$ tensor product of $\hten$.  We use $\hten^{\odot q}$ to denote the symmetric tensor product.

For integers $q \ge 1$, let ${\cal H}_q$ be the $q^{th}$ Wiener chaos of $X$, that is, the closed linear subspace of $L^2(\Omega)$ generated by the random variables $\{ H_q(X(h)), h \in \hten, \|h \|_\hten = 1 \}$, where $H_q(x)$ is the $q^{th}$ Hermite polynomial, defined as
$$H_q (x) = {(-1)^q}e^{\frac{x^2}{2}}\frac{d^q}{dx^q}e^{-\frac{x^2}{2}}.$$  In particular, $H_1(x) =x$, $H_2(x) = x^2-1$, and $H_3(x) = x^3 - 3x$.  For $q \ge 1$, it is known that the map 
\beq{Hmap} I_q(h^{\otimes q}) = H_q(X(h))\eeq
provides an isometry between the symmetric product space $\hten^{\odot q}$ (equipped with the modified norm $\frac{1}{\sqrt{q!}}\| \cdot\|_{\hten^{\otimes q}}$) and ${\cal H}_q$.  By convention, ${\cal H}_0 = \mathbb{R}$ and $I_0(x) = x$.

\subsection{Elements of Malliavin Calculus}
Following is a brief description of some identities that will be used in the paper.  The reader may refer to \cite{NoNu} for a brief survey, or to \cite{Nualart} for detailed coverage of this topic.  Let $\cal S$ be the set of all smooth and cylindrical random variables of the form $F = g(X(\phi_1), \dots, X(\phi_n))$, where $n \ge 1$; $g: {\mathbb R}^n \to {\mathbb R}$ is an infinitely differentiable function with compact support, and $\phi_i \in \hten$. The Malliavin derivative of $F$ with respect to $X$ is the element of $L^2(\Omega, \hten)$ defined as
$$DF = \sum_{i=1}^n \frac{\partial g}{\partial x_i}(X(\phi_1), \dots, X(\phi_n)) \phi_i.$$
In particular, $DX(h) = h$.  By iteration, for any integer $q >1$ we can define the $q^{th}$ derivative $D^qF$, which is an element of $L^2(\Omega, \hten^{\odot q})$.

For any integer $q \ge 1$ and real number $p \ge 1$, let ${\mathbb D}^{q,p}$ denote the closure of $\cal S$ with respect to the norm $\| \cdot \|_{{\mathbb D}^{q,p}}$ defined as 
$$\| F \|_{{\mathbb D}^{q,p}}^p = {\mathbb E}\left[ |F|^p\right] + \sum_{i=1}^q {\mathbb E}\left[ \| D^iF \|_{\hten^{\otimes i}}^p \right].$$
We denote by $\delta$ the Skorohod integral, which is defined as the adjoint of the operator $D$.  This operator is also referred to as the divergence operator in \cite{Nualart}.  A random element $u \in L^2(\Omega, \hten)$ belongs to the domain of $\delta$, Dom $\delta$, if and only if,
$$\left| {\mathbb E}\left[ \left< DF, u\right>_\hten\right] \right| \le c_u \sqrt{E[F^2]}$$
for any $F \in {\mathbb D}^{1,2}$, where $c_u$ is a constant which depends only on $u$.  If $u \in $ Dom $\delta$, then the random variable $\delta (u) \in L^2(\Omega)$ is defined for all $F \in {\mathbb D}^{1,2}$ by the duality relationship,
$${\mathbb E}\left[ F\delta(u) \right] = {\mathbb E}\left[ \left< DF, u \right>_\hten \right].$$
This is sometimes called the Malliavin integration by parts formula.  We iteratively define the multiple Skorohod integral for $q \ge 1$ as $\delta (\delta^{q-1}(u))$, with $\delta^0(u) = u$.  For this definition we have,
\beq{Duality}{\mathbb E}\left[ F\delta^q(u) \right] = {\mathbb E}\left[ \left< D^qF, u \right>_{\hten^{\otimes q}} \right],\eeq
where $u \in$ Dom $\delta^q$ and $F \in {\mathbb D}^{q,2}$.  Moreover, if $h \in \hten^{\odot q}$, then we have $\delta^q (h) = I_q(h)$.

For $f \in \hten^{\odot p}$ and $g\in \hten^{\odot q}$, the following integral multiplication formula holds:
\beq{intmult} \delta^p (f) \delta^q(g) = \sum_{r=0}^{p\wedge q} r! \binom{p}{r}\binom qr \delta^{p+q-2r}(f \otimes_r g),\eeq
where $\otimes_r$ is the contraction operator (see, {\em e.g.,} \cite{Nualart}, Sec. 1.1).  

We will use the Meyer inequality for the Skorohod integral, (see, for example Prop. 1.5.7 of \cite{Nualart}).  Let ${\mathbb D}^{k,p}(\hten^{\otimes k})$ denote the corresponding Sobolev space of $\hten^{\otimes k}$-valued random variables.  Then for $p \ge 1$ and integers $k \ge q \ge 1$, we have,
\beq{Meyer}\| \delta^q(u) \|_{{\mathbb D}^{k-q, p}} \le c_{k,p} \| u \|_{{{\mathbb D}^{k,p}}(\hten^{\otimes q})}\eeq
for all $u \in {\mathbb D}^{k,p}(\hten^{\otimes k})$ and some constant $c_{k,p}$.

The following three results are well known, and will be used extensively in this paper.  The reader may refer to \cite{NoNu} and \cite{Nualart} for details.
\begin{lemma}  Let $q \ge 1$ be an integer.

\begin{enumerate}[(a)]
\item Assume $F \in {\mathbb D}^{q,2}$, $u$ is a symmetric element of {\em Dom}  $\delta^q$, and $\left< D^rF, \delta^j (u) \right>_{\hten^{\otimes r}} \in L^2(\Omega, \hten^{\otimes q-r-j})$ for all $0 \le r+j \le q.$  Then $\left< D^rF, u \right>_{\hten^{\otimes r}} \in \,${\em Dom }$\delta^r$ and
$$ F \delta^q(u) = \sum_{r=0}^q \binom{q}{r} \delta^{q-r}\left(\left< D^rF, u \right>_{\hten^{\otimes r}}\right).$$

\item  Suppose that $u$ is a symmetric element of ${\mathbb D}^{j+k,2}(\hten^{\otimes j})$.  Then we have,
$$D^k \delta^j (u) = \sum_{i=0}^{j \wedge k} \binom{k}{i} \binom{j}{i} i! \delta^{j-i}\left(D^{k-i}u\right).$$

\item  Let $u, v$ be symmetric functions in ${\mathbb D}^{2q, 2}(\hten^{\otimes q})$.  Then
$${\mathbb E}[\delta^q(u) \delta^q(v) ] = \sum_{i=0}^q \binom{q}{i}^2 {\mathbb E}\left[ \left< D^{q-i}u, D^{q-i}v\right>_{\hten^{\otimes(2q-i)}}\right].$$
In particular,
$$\left\| \delta^q (u) \right\|^2_{L^2(\Omega)} = {\mathbb E}\left[ \delta^q(u)^2\right] \le
\sum_{i=0}^q \binom{q}{i}^2 {\mathbb E}\left[ \left\| D^{q-i}u\right\|_{\hten^{\otimes(2q-i)}}^2\right].$$ 
\end{enumerate}\end{lemma}
 
\bigskip

\subsection{Criteria for convergence in the Skorohod space ${\mathbf D}[0,\infty)$}
\begin{definition} 
Assume $F_n$ is a sequence of $d-$dimensional random variables defined on a probability space $(\Omega, {\cal F}, P)$, and $F$ is a $d-$dimensional random variable defined on $(\Omega, {\cal G}, P)$, where ${\cal F} \subset {\cal G}$.  We say that $F_n$ {\em converges stably} to $F$ as $n \to \infty$, if, for any continuous and bounded function $f:\,{\mathbb R}^d \to {\mathbb R}$ and ${\mathbb R}$-valued, ${\cal F}-$measurable random variable $Z$, we have
$$\lim_{n\to \infty} {\mathbb E} \left(f(F_n) Z\right) = {\mathbb E}\left(f(F) Z\right).$$
\end{definition}

The following central limit theorem first appeared in \cite{NoNu}, and the present multi-dimensional version was proved in \cite{HaNu}.  This will be the main theoretical tool of the paper.
\begin{theorem} 
Let $q \ge 1$ be an integer, and suppose that $F_n$ is a sequence of random variables in $\mathbb{R}^d$ of the form $F_n = \delta^q (u_n) = \left( \delta^q (u_n^1), \dots, \delta^q(u_n^d) \right)$, for a sequence of $\mathbb{R}^d-$valued symmetric functions $u_n$ in $\mathbb{D}^{2q, 2q}(\hten^{\otimes q})$.  Suppose that the sequence $F_n$ is bounded in $L^1(\Omega, \hten)$ and that:
\begin{enumerate}[(a)]
\item  $\left< u_n^j, \bigotimes_{\ell =1}^m (D^{a_\ell}F_n^{j_\ell}) \otimes h \right>_{\hten^{\otimes q}}$ converges to zero in $L^1(\Omega)$ for all  integers $1 \le j, j_\ell \le d$, all integers $1 \le a_1, \dots, a_m, r \le q-1$  such that $a_1 + \cdots + a_m + r = q$; and all $h \in \hten^{\otimes r}$.
\item  For each $1 \le i,j \le d$, $\left< u_n^i, D^qF_n^j\right>_{\hten^{\otimes q}}$ converges in $L^1(\Omega, \hten)$ to a random variable $s_{ij}$, such that the matrix $\Sigma := \left( s_{ij} \right)_{d \times d}$ is nonnegative definite (that is, $\lambda^T \Sigma \lambda \ge 0$ for all nonzero $\lambda \in {\mathbb R}^d$).
\end{enumerate}
Then $F_n$ converges stably to a random variable in ${\mathbb R}^d$ with conditional Gaussian law ${\cal N} (0, \Sigma)$ given $X$. \end{theorem}

\medskip 
\begin{remark}
Conditions (a) and (b) mean that for $q \ge 1$, some combinations of lower-order derivative products are negligible.  In particular, for $q=3$, then the following scalar products will converge to zero in $L^1(\Omega, \hten)$:
\begin{itemize}
\item $\left< u_n^i, h\right>_{\hten^{\otimes 3}}$ for all $h \in \hten^{\otimes 3}$.
\item $\left< u_n^i, DF_n^j \otimes h \right>_{\hten^{\otimes 3}} $ for all $h \in \hten^{\otimes 2}$ and all $j$ (including $i=j$).
\item $\left< u_n^i, D^2F_n^j \otimes h \right>_{\hten^{\otimes 3}} $ for all $h \in \hten$ and all $j$.
\item $\left< u_n^i, D^2F_n^j \otimes DF_n^k \right>_{\hten^{\otimes 3}}$ and $\left< u_n^i, DF_n^j \otimes DF_n^k\otimes DF_n^\ell \right>_{\hten^{\otimes 3}}$for all $1 \le k,j,\ell \le d$.
\end{itemize}
Only the $3^{rd}$-order derivative products $\left< u_n^i, D^3F_n^j\right>_{\hten^{\otimes 3}}$ converge to a nontrivial random variable.   \end{remark}

\medskip
\begin{remark}  It suffices to impose condition (a) for $h \in {\cal S}_0$, where ${\cal S}_0$ is a total subset of $\hten^{\otimes r}$.\end{remark}

\medskip 
For the main result of this paper, we will also require that the sequence $F_n$ satisfies a relative compactness condition in order to establish convergence in the Skorohod space ${\mathbf D}[0,\infty)$.
\begin{corollary}
Suppose $\{ G_n(t), t\ge 0\}$ is a sequence of ${\mathbb R}$-valued processes of the form $G_n(t) = \delta^q\left(u_n(t)\right)$, where $u_n(t)$ is a sequence of symmetric functions in ${\mathbb D}^{2q,2q}(\hten^{\otimes q})$.  Assume that for any finite set of times
$\{ 0=t_0 < t_1 <\dots<t_d\}$, the sequence
$$\left( G_n(t_1)-G_n(t_0), \dots, G_n(t_d) - G_n(t_{d-1})\right)$$ 
satisfies Theorem 2.3; where the $d\times d$ matrix $\Sigma$ is diagonal with entries $s^2(t_i) - s^2(t_{i-1})$. Suppose further that there exist real numbers $C>0$, $\gamma > 0$, and $\beta >1$ such that for each $n$ and for any $0 \le t_1 < t < t_2$, we have 
$${\mathbb E}\left[ \left| G_n(t)-G_n(t_1)\right|^\gamma\left| G_n(t_2)-G_n(t)\right|^\gamma\right] \le C\left(\frac{\lfloor nt_2 \rfloor - \lfloor nt_1 \rfloor}{n}\right)^\beta.$$
Then the family of stochastic processes $\{ G_n, n \ge 1\}$ converges as $n \to \infty$ to the process $G = \{G_t, t\ge 0\}$, where $G(t)$ is a Gaussian random variable with mean zero and variance $s^2(t)$.  Equivalently, we can say that $G_n(t) \law \sqrt{s^2(t)}~Z$ as $n \to \infty$, where $Z \sim {\cal N}(0,1)$.\end{corollary}

This convergence criteria in ${\mathbb D}$ is well known (see, {\em e.g}, \cite{Billingsley}, Thm. 13.5). 

\section{Convergence of the Stratonovich integral}
\subsection{Covariance conditions}
Consider a Gaussian stochastic process $X := \{ X_t, t \ge 0\}$ with covariance function ${\mathbb E}\left[ X_s X_t\right] = R(s,t)$. Assume $R(s,t)$ satisfies the following bounds:  for any $T > 0$, $0 < s \le 1$, and $s \le r,t \le T$:
\begin{enumerate}[(i)]
\item  $ {\mathbb E}\left[ (X_t-X_{t-s})^2\right] \le C_1s^{\frac 13}$, for a positive constant $C_1$.
\item  If $t > s$,
$$\left|{\mathbb E}\left[ X_t^2 - X_{t-s}^2 \right]\right| \le C_2 s^{\frac{1}{3}+\theta} (t-s)^{-\theta}$$ for some $C_2$ and $1/2 < \theta < 1$.
\item  For $t \ge 4s$,
$$\left|{\mathbb E}\left[ (X_t - X_{t-s})^2 - (X_{t-s} - X_{t-2s})^2\right]\right| \le C_3 s^{\frac 13 + \nu}(t-2s)^{-\nu}$$
for some constants $C_3$ and $\nu > 1$.
\item  There is a constant $C_4$ and a real number $\lambda \in (\frac 16, \frac 13]$ such that $$\left| {\mathbb E}\left[ X_r (X_t - X_{t-s})\right]\right| \le \begin{cases}  C_4s\left( (t-s)^{\lambda-1}+|t-r|^{\lambda-1}\right)&\text{ if } |t-r| \ge 2s \text{ and } t\ge 2s\\
C_4s^{\lambda} &\text{otherwise}\end{cases}$$
\item  There is a constant $C_5$ and a real number $\gamma >  1$ such that for $t\wedge r \ge 2s$ and $|t-r| \ge 2s$, $$\left|{\mathbb E}\left[ (X_t - X_{t-s}) (X_r - X_{r-s})\right]\right| \le  C_5 s^{\frac{1}{3}+\gamma}|t-r|^{-\gamma} .$$
\item  For integers $n >0$ and integers $0\le j,k \le nT,$ define $\beta_n(j,k) := {\mathbb E}\left[ (X_{\frac{j+1}{n}} - X_{\frac{j}{n}})(X_{\frac{k+1}{n}} - X_{\frac{k}{n}})\right]$.  Then for each real number $0 \le t \le T$,
\beq{eta_def}\lim_{n \to \infty} \sum_{j,k=0}^{\Nt -1} \beta_n(j,k)^3 = \eta(t),\eeq
where $\eta(t)$ is a continuous and nondecreasing function with $\eta(0)=0$.  As we will see, $\eta(t)$ is comparable to the `cubic variation' $[X,X,X]_t$ discussed in \cite{GNRV} and \cite{NoRevSwan}.  As described in \cite{NoRevSwan}, these terms are related by Theorem 10 of \cite{NuOrtiz}.  \end{enumerate}

In particular, it can be shown that the above conditions are satisfied by fBm with Hurst parameter $H=1/6$.  In Section 4 we will show additional examples.

\medskip
In addition to conditions (i) - (vi) on $X$, we will also assume the following condition (0) on the test function $f$:

\begin{enumerate}[(0)]
\item  Let $f:{\mathbb R}\to{\mathbb R}$ be a ${\cal C}^\infty$ function, such that $f$ and all its derivatives have at most polynomial growth.  
\end{enumerate}

\medskip 
Consider a uniform partition of $[0, \infty) $ with increment length $1/n$.  The Stratonovich integral of $f'(W)$ will be defined as the limit in probability of the sequence (see \cite{NoRevSwan}):
\beq{Phi_n} \Phi_n^X(t) := \frac 12\sum_{j=0}^{\Nt -1} \left[f'(X_{\frac jn}) + f'(X_{\frac{j+1}{n}})\right] \Delta X_{\frac jn} \eeq
where $\Delta X_{\frac jn} = X_{\frac{j+1}{n}} - X_{\frac jn}$. 

\bigskip
The following is the major result of this section.

\begin{theorem}
Let $f$ be a real function satisfying condition (0), and let $X = \{X_t, t \ge 0\}$ be a Gaussian process satisfying conditions  (i) through (vi).  Then:
$$\left( X_t, \Phi_n^X(t) \right) \law \left( X_t, f(X_t) - f(X_0) + \frac{\sqrt{6}}{12}\int_0^t f^{(3)}(X_s)~ dB_s \right)$$
as $n \to \infty$ in the Skorohod space ${\mathbf D}[0, \infty)$, where $B = \{B_t, t\ge 0\}$ is a scaled Brownian motion, independent of $X$, and with variance ${\mathbb E}\left[B_t^2\right] = \eta(t)$ for the function $\eta$ defined in condition (vi). \end{theorem}

\medskip
The proof follows from Theorem 2.3 and Corollary 2.6, and is given in a series of lemmas.  Following is an outline of the proof. After a preliminary technical lemma, we use a Taylor expansion to decompose \[\Phi_n^X(t) = f(X_t) - f(X_0) + \frac{1}{12}F_n(t) + \Delta_n(t).\]
We first show that $\Delta_n(t) \stackrel{\cal P}{\longrightarrow} 0$ as $n \to\infty$; then we show that $F_n(t)$ satisfies Theorem 2.3.  Next we show that $F_n(t) + \Delta_n(t)$ is relatively compact in the sense of Corollary 2.6, and the result follows.

\medskip 
Introduce the following notation, which is similar to that of \cite{HaNu} and \cite{NoNu}.  Let $\varepsilon_t = {\mathbf 1}_{[0,t]}$; $\partial_{\frac jn} = {\mathbf 1}_{\left[{\frac jn},\frac{j+1}{n}\right]}$; let $\Xhat = \frac 12\left(X_{\frac jn} + X_{\frac{j+1}{n}}\right)$; and $\hat{\varepsilon}_{\frac jn} = \frac 12 (\varepsilon_{\frac jn} + \varepsilon_{\frac{j+1}{n}})$.  In the following, the term $C$ represents a generic positive constant, which may change from line to line.  The constant $C$ may depend on $T$ and the constants in conditions (0) and (i) - (vi) listed above.  By the isometry between the space generated by $X$ and the Hilbert space $\hten$, we will use the terms ${\mathbb E}[X_s X_t] = \left< \varepsilon_s, \varepsilon_t\right>_\hten$ interchangeably.

We begin with the following technical results, which follow from conditions (i) through (v).
\begin{lemma}  Assume $\{X_t, 0\le t \le T\}$ satisfies conditions (i), (ii), (iv) and (v).  For integers $n \ge 1$, $r \ge 1$ and integers $0 \le a < b < c \le \lfloor nT \rfloor$, there exists a constant $C > 0$, which does not depend on $a,b,c$ or $r$, such that: 
\begin{enumerate}[(a)]
\item \[\sup_{0\le j,k\le\lfloor nT \rfloor} \left|\left< \partial_{\frac jn}, \partial_{\frac kn}\right>_\hten\right| \le Cn^{-\frac 13};\quad \text{ and } \; \sup_{0\le u \le T}\sup_{0\le j \le \lfloor nT \rfloor} \left| \left< \varepsilon_u, \partial_{\frac jn}\right>_\hten\right| \le Cn^{-\lambda}.\]
\item \begin{align} &\sum_{j=a}^b \left| \left<\hat{\varepsilon}_{\frac jn} , \partial_{\frac jn }\right>_\hten\right|
\le Cn^{-\frac 13}(b-a+1)^{1-\theta};\;\text{ and  }\\
&\sum_{j=a}^b \left|\left<\hat{\varepsilon}_{\frac jn} , \partial_{\frac jn }\right>_\hten\right|^r \le Cn^{-\frac r3}\; \text{ for }\;r>1.\end{align}

\item For $0\le u,v \le T$, 
\begin{align}
&\sum_{j = a}^b \left|\left< \varepsilon_u , \partial_{\frac jn }\right>_\hten\right|
\le C;\;\text{ and }\\ 
&\sum_{j=a}^b \left|\left<\varepsilon_u , \partial_{\frac jn }\right>_\hten\left<\varepsilon_v , \partial_{\frac jn }\right>_\hten\right| \le Cn^{-2\lambda}.\end{align}
\item  \beq{L32c}\sum_{j,k=a}^b \left|\left<\partial_{\frac jn}, \partial_{\frac kn}\right>_\hten \right|^r \le C(b-a+1)n^{-\frac r3}.\eeq
\item  \beq{L32d}\sum_{k=b+1}^c \sum_{j=a}^b \left|\left< \partial_{\frac jn}, \partial_{\frac kn}\right>_\hten \right|^r \le C(c-b)^\epsilon n^{-\frac r3}\eeq where $\epsilon = \max\{ 1-\theta, 2-\gamma\}$.
\end{enumerate}\end{lemma}
\begin{proof}
We may assume $a = 0$.  For part (a), the first inequality follows immediately from condition (i) and Cauchy-Schwarz; and the second inequality is just a restatement of condition (iv).  
For (b), applying condition (i) for $j=0$ and condition (ii) for $j \ge 1$, we have:
\begin{align*}\sum_{j=0}^b \left| {\mathbb E}\left[ X_{\frac{j+1}{n}}^2 - X_{\frac jn}^2\right]\right| &\le Cn^{-\frac 13}\sum_{j=1}^b j^{-\theta} + Cn^{-\frac 13}\\
&\le Cn^{-\frac 13}\int_0^{b} u^{-\theta} du + Cn^{-\frac 13}\\
&\le Cn^{-\frac 13}(b+1)^{1-\theta}.\end{align*}
Then if $r \ge 2$, 
\begin{align*}\sum_{j=0}^b \left| {\mathbb E}\left[ X_{\frac{j+1}{n}}^2 - X_{\frac jn}^2\right]\right|^r &\le Cn^{-\frac r3}\sum_{j=1}^b j^{-r\theta} + Cn^{-\frac r3}\\
&\le Cn^{-\frac r3}\end{align*}
because $\theta > 1/2$ implies $j^{-r\theta}$ is summable.

For (c), define the set $J_c = \{ j: 0\le j \le b,\, j=0\; or \;|j-nu| < 2\; or\; |j-nv| < 2\}$, and note that $|J_c| \le 7$.  Then we have by (a) and condition (iv),
\begin{align*}
\sum_{j=0}^b \left| \left<\varepsilon_u , \partial_{\frac jn }\right>_\hten\right| & \le \sum_{j\in J_c} Cn^{-\lambda} + Cn^{-\lambda}\sum_{j \notin J_c} \left(j^{\lambda-1} + |j-nu|^{\lambda-1}\right)\\
&\le Cn^{-\lambda} + Cn^{-\lambda}(b+1)^{\lambda} \le C, \end{align*}
and
\begin{align*}
\sum_{j=0}^b \left|\left<\varepsilon_u , \partial_{\frac jn }\right>_\hten\left<\varepsilon_v , \partial_{\frac jn 	}\right>_\hten\right| &\le \sum_{j\in J_c} Cn^{-2\lambda} + Cn^{-2\lambda}\sum_{j \notin J_c} \left(j^{2\lambda-2} + |j-nu|^{2\lambda-2}+|j-nv|^{2\lambda-2}\right)\\
&\le Cn^{-2\lambda} + Cn^{-2\lambda}\sum_{p=1}^\infty p^{2\lambda-2}\\
&\le Cn^{-2\lambda}\end{align*}
because $\lambda \le 1/3$.  

For (d), define the set: $J_d =\left\{ j,k: j\wedge k < 1 \;or\; |j-k| < 2\right\},$ and note that $|J_d| \le 6(b+1)$.  Then we have by (a) and condition (v)
\begin{align*}
\sum_{j,k=0}^b \left|\left<\partial_{\frac jn}, \partial_{\frac kn}\right>_\hten \right|^r
&\le \sup_{j,k} \left|\left<\partial_{\frac jn}, \partial_{\frac kn}\right>\right|^{r-1}\sum_{j,k = 0}^b \left|\left<\partial_{\frac jn}, \partial_{\frac kn}\right>_\hten \right|\le Cn^{-\frac{r-1}{3}}\sum_{j,k = 0}^b \left|\left<\partial_{\frac jn}, \partial_{\frac kn}\right>_\hten \right|\\
&\le Cn^{-\frac{r-1}{3}}\left( \sum_{(j,k) \in J_d} n^{-\frac 13} + n^{-\frac 13}\sum_{(j,k) \notin J_d} |j-k|^{-\gamma}\right)\\
&\le C(b+1)n^{-\frac r3}.\end{align*}
In particular, if $r=3$ and $b = \lfloor nt\rfloor -1$ (as in condition (vi)), the sum converges absolutely, and the sum vanishes if $r > 3$.

For (e), we consider the maximal case, which occurs when $a=0$:
\begin{align*}
\sum_{k=b+1}^c \sum_{j=0}^b \left|\left< \partial_{\frac jn}, \partial_{\frac kn}\right>_\hten \right|^r
&=\sum_{k=b+1}^c \left|\left< \partial_{\frac 0n}, \partial_{\frac kn}\right>_\hten \right|^r + \sum_{k=b+1}^c \sum_{j=1}^{b-1} \left|\left< \partial_{\frac jn}, \partial_{\frac kn}\right>_\hten \right|^r +\sum_{k=b+1}^c \left|\left< \partial_{\frac bn}, \partial_{\frac kn}\right>_\hten \right|^r. 
\end{align*}
Note that $\partial_{\frac 0n} = \varepsilon_{\frac 1n}$. By part (c) and condition (v), respectively, this is
\begin{align*} &\le \sum_{k=b+1}^c \left|\left< \varepsilon_{\frac 1n}, \partial_{\frac kn}\right>_\hten \right|^r +
Cn^{-\frac r3}\sum_{k=b+1}^c \sum_{j=1}^{b-1} (k-j)^{-\gamma}  + Cn^{-\frac r3}\sum_{k=b+1}^c (k-b)^{-\gamma} \\
&\le Cn^{-\frac r3}(c-b)^{1-\theta} + Cn^{-\frac r3}(c-b)^{2-\gamma} + Cn^{-\frac r3}\\
&\le Cn^{-\frac r3}(c-b)^\epsilon,
\end{align*}
where $\epsilon = \max\{ 1-\theta, 2-\gamma\} < 1$.
\end{proof}

\subsection{Taylor expansion of $\Phi_n^X(t)$}
The details of this expansion were mainly inspired by Lemma 5.2 of \cite{NoRevSwan}.  We begin with the telescoping series,
$$f(X_t) = f(0) + f(X_t) - f(X_{\frac{\Nt}{n}}) + \sum_{j=0}^{\Nt -1} \left[f(X_{\frac{j+1}{n}}) - f(X_{\frac{j}{n}})\right] .$$
By continuity of $f$ and $X$, we know that for large $n$, $f(X_t) - f(X_{\frac{\Nt}{n}}) \to 0$ uniformly on compacts in probability (ucp), so this term may be neglected.
For each $j$, we use a Taylor expansion of order 6 with residual term.  Let
$ h_j := \frac 12 \left[ X_{\frac{j+1}{n}} - X_{\frac{j}{n}}\right].$
Then:
\begin{align*}
f(X_{\frac{j+1}{n}}) - f(X_{\frac{j}{n}}) &= \left( f(\hat{X}_{\frac jn} + h_j) - f(\hat{X}_{\frac jn})\right) - \left( f(\hat{X}_{\frac jn} - h_j) - f(\hat{X}_{\frac jn})\right)\\
&= \sum_{k=1}^6 f^{(k)}(\hat{X}_{\frac jn})\frac{h_j^k}{k!}+ R_n^+(j) - \left(\sum_{k=1}^6 (-1)^kf^{(k)}(\hat{X}_{\frac jn})\frac{h_j^k}{k!}+ R_n^-(j)\right)\\
&=f'(\hat{X}_{\frac jn})\Delta X_{\frac{j}{n}} + \frac{1}{24}f^{(3)}(\hat{X}_{\frac jn})\Delta X_{\frac{j}{n}}^3 + \frac{1}{2^4 5!}f^{(5)}(\hat{X}_{\frac jn})\Delta X_{\frac jn}^5 +R_n^+(j) - R_n^-(j)
\end{align*}
where $\Delta X_{\frac{j}{n}} = 2h_j$; and $R_n^+(j), R^-_n(j)$ are Taylor series remainder terms of order 7.  Next we compute:
\begin{align*}
\frac{f'(X_{\frac{j+1}{n}}) + f'(X_{\frac{j}{n}})}{2} - f'(\hat{X}_{\frac jn}) &= \frac{1}{2}\left( f'(\hat{X}_{\frac jn} + h_j) - f'(\hat{X}_{\frac jn})\right) + 
\frac{1}{2}\left( f'(\hat{X}_{\frac jn} - h_j) - f'(\hat{X}_{\frac jn})\right)\\
&= \frac{1}{2}\sum_{k=1}^5 f^{(1+k)}(\hat{X}_{\frac jn})\frac{h_j^k}{k!} + K_n^+(j) +\frac{1}{2}\sum_{k=1}^5 (-1)^k f^{(1+k)}(\hat{X}_{\frac jn})\frac{h_j^k}{k!} + K_n^-(j)\\
&=\frac{1}{8}f^{(3)}(\hat{X}_{\frac jn}) \Delta X_{\frac{j}{n}}^2 + \frac{1}{2^4 4!}f^{(5)}(\hat{X}_{\frac jn}) \Delta X_{\frac{j}{n}}^4 + \frac{1}{2}\left( K_j^+ + K_j^-\right),
\end{align*}
where $K^+_j, K^-_j$ are remainder terms of order 6.  Combining the two equations, we obtain
\begin{align*}f(X_{\frac{j+1}{n}}) - f(X_{\frac{j}{n}}) &= \frac{f'(X_{\frac{j+1}{n}}) + f'(X_{\frac{j}{n}})}{2}\Delta X_{\frac{j}{n}} - \frac{1}{12}f^{(3)}(\hat{X_j})\Delta X_{\frac{j}{n}}^3 - 
 \frac{4}{2^5 5!} f^{(5)}(\hat{X_j})\Delta X_{\frac{j}{n}}^5\\&\quad+ R_n^+(j) - R_n^-(j) -\frac{1}{4}\left[K_n^+(j) + K_n^-(j)\right]\Delta X_{\frac{j}{n}}.\end{align*}

Our first task is to show that the $f^{(5)}$ terms and the remainder term vanish in probability.
\begin{lemma}For each integer $n \ge 1$ and real numbers $0 \le t_1 < t_2\le T $, 
\beq{f5tight}{\mathbb E}\left[\left(\sum_{j=\lfloor nt_1\rfloor}^{\lfloor nt_2 \rfloor -1} f^{(5)}(\hat{X}_{\frac jn})\Delta X_{\frac jn}^5\right)^2\right] \le Cn^{-\frac 43}\left(
\lfloor nt_2 \rfloor - \lfloor nt_1 \rfloor\right).\eeq\end{lemma}

The proof of this lemma is technical, and is deferred to Section 5.

\medskip
\begin{lemma}  For integers $n \ge 1$, let \[Z_n(t) =\sum_{j=0}^{\Nt -1} \left[R_n^+(j) - R_n^-(j) + \frac 14 \left( K^+_n(j) + K_n^-(j) \right)\Delta X_{\frac jn}\right].\]  Then for real numbers $0 \le t_1 < t_2 \le T$, we have
\beq{Zntight}{\mathbb E}\left[ \left(Z_n(t_2) - Z_n(t_1)\right)^2\right] \le Cn^{-\frac 73}\left( \lfloor nt_2 \rfloor - \lfloor nt_1 \rfloor\right)^2.\eeq\end{lemma}

\begin{proof}
We may assume $t_1 = 0$.  Observe that each term in the sum $Z_n(t)$ has the form
$$Cf^{(7)}(\xi_j)\Delta X_{\frac jn}^7,$$
where $\xi_j$ is an intermediate value between $X_{\frac jn}$ and $X_{\frac{j+1}{n}}$.  Using the H\"older inequality, for each $0 \le j,k < \lfloor nt_2\rfloor$ we have
$${\mathbb E}\left[ f^{(7)}(\xi_j)f^{(7)}(\xi_k)\Delta X_{\frac jn}^7\Delta X_{\frac kn}^7\right]$$
$$\le \left( \sup_{0<u<1}{\mathbb E}\left[ f^{(7)}(uX_{\frac jn} + (1-u)X_{\frac{j+1}{n}})^4\right]~\sup_{0<v<1}{\mathbb E}\left[ f^{(7)}(vX_{\frac kn} + (1-v)X_{\frac{k+1}{n}})^4\right]~{\mathbb E}\left[\left|\Delta X_{\frac jn}^7\right|^4\right]~{\mathbb E}\left[\left|\Delta X_{\frac kn}^7\right|^4\right]\right)^{\frac 14}.$$
By condition (0), the first two terms are bounded.  By condition (i), ${\mathbb E}\left[ \Delta X_{\frac jn}^2\right] \le C_1n^{-\frac 13}$; and we have by the Gaussian moments formula that
$${\mathbb E}\left[ \Delta X_{\frac jn}^{28}\right] \le 27!!\left(C_1n^{-\frac 13}\right)^{14},$$
hence it follows that
$$C\sum_{j,k=0}^{\lfloor nt_2 \rfloor -1} {\mathbb E}\left[f^{(7)}(\xi_j)f^{(7)}(\xi_k)\Delta X_{\frac jn}^7\Delta X_{\frac kn}^7\right]
\le C\sum_{j,k=0}^{\lfloor nt_2 \rfloor -1} n^{-\frac{7}{3}} \le C\lfloor nt_2\rfloor^2 n^{-\frac 73}.$$
\end{proof}

\subsection{Malliavin calculus representation of $3^{rd}$ order term}
From Lemmas 3.3 and 3.4, it follows that as $n$ becomes large, then $f(X_t)$ behaves asymptotically as
$$f(X_0) + \frac 12\sum_{j=0}^{\Nt -1} \left[f'(X_{\frac{j+1}{n}}) + f'(X_{\frac jn})\right]\Delta X_{\frac jn} -\frac{1}{12}\sum_{j=0}^{\Nt -1} f^{(3)}(\hat{X}_{\frac jn})\Delta X_{\frac jn}^3.$$  We now turn to the $f^{(3)}(\hat{X}_{\frac jn})$ term in the Taylor expansion. 

\medskip
\begin{remark}  It may happen that the upper bound of condition (v) is such that
$$\eta(t) \le \lim_{n\to \infty} \sum_{j,k=0}^{\Nt -1}\left|\beta_n(j,k)^3\right| =0$$
for all $t$, which implies
$$\lim_n \,{\mathbb E}\left[\left(\sum_{j=0}^{\Nt -1} f^{(3)}(\hat{X}_{\frac jn})\Delta X_{\frac jn}^3\right)^2\right] = 0$$
for any function $f$ satisfying condition (0).  In this case, $\Phi_n^X(t)$ converges in probability and we have the change-of-variable formula \req{gnrv}.  Indeed, this corresponds to the case of zero cubic variation discussed in \cite{GNRV}.  In the rest of this section, we will assume that $\eta(t)$ is non-trivial. \end{remark}

\medskip
Consider the $3^{rd}$ Hermite polynomial $H_3(y) = y^3 - 3y$.  For $y=\frac{\Delta X}{\|\Delta X\|_{L^2}},$ it follows that
\begin{align*} \sum_{j=0}^{\Nt -1} f^{(3)}(\hat{X}_{\frac jn}) \Delta X_{\frac jn}^3 = &
\sum_{j=0}^{\Nt -1}  \|\Delta X_{\frac jn}\|_{L^2}^3f^{(3)}(\hat{X}_{\frac jn})H_3\left(\frac{\Delta X_{\frac jn}}{\|\Delta X_{\frac jn}\|_{L^2}}\right) \\
&+3\sum_{j=0}^{\Nt -1}  \|\Delta X_{\frac jn}\|_{L^2}^2 f^{(3)}(\hat{X}_{\frac jn})\Delta X_{\frac jn}\end{align*}
The second term is dealt with in the next lemma.  The proof is technical, and is deferred to Section 5. 
\begin{lemma}  For integers $n \ge 1$ and integers $0 \le a < b \le nT$, 
$${\mathbb E} \left|\sum_{j=a}^{b -1}  \|\Delta X_{\frac jn}\|_{L^2}^2 f^{(3)}(\hat{X}_{\frac jn})\Delta X_{\frac jn}\right| \le Cn^{-\frac 13}.$$
\end{lemma}

\medskip
Next, we consider the $H_3$ term.  By \req{Hmap} and Lemma 2.1.a we have
\begin{align*}
\sum_{j=0}^{\Nt -1}  \|\Delta X_{\frac jn}\|_{L^2}^3f^{(3)}(\hat{X}_{\frac jn})H_3\left(\frac{\Delta X_{\frac jn}}{\|\Delta X_{\frac jn}\|_{L^2}}\right) &= \sum_{j=0}^{\Nt -1} f^{(3)}(\hat{X}_{\frac jn})\delta^3\left(\partial_{\frac jn}^{\otimes 3}\right) \\
&=\sum_{j=0}^{\Nt -1} \delta^3\left(f^{(3)}(\hat{X}_{\frac jn})\partial_{\frac jn}^{\otimes 3}\right)  + 3\delta^2\left(\left<Df^{(3)}(\hat{X}_{\frac jn}), \partial_{\frac jn}^{\otimes 3}\right>_\hten \right)\\
&\quad +3\delta\left(\left<D^2f^{(3)}(\hat{X}_{\frac jn}), \partial_{\frac jn}^{\otimes 3}\right>_{\hten^{\otimes 2}} \right) +\left<D^3f^{(3)}(\hat{X}_{\frac jn}), \partial_{\frac jn}^{\otimes 3}\right>_{\hten^{\otimes 3}}\\
&= \sum_{j=0}^{\Nt -1} \delta^3\left(f^{(3)}(\hat{X}_{\frac jn})\partial_{\frac jn}^{\otimes 3}\right) + P_n(t).
\end{align*}

As $n \to \infty$, we show that the term $P_n(t)$ vanishes in probability.

\begin{lemma}For integers $n \ge 1$ and real numbers $0 \le t_1 < t_2 \le T$,
$${\mathbb E}\left[P_n(t)^2\right] \le C\left(\lfloor nt_2 \rfloor - \lfloor nt_1 \rfloor\right)n^{-\frac 43}.$$\end{lemma}

\begin{proof}  We may assume $t_1 =0$.  We want to show
\beq{L3e1}{\mathbb E}\left[\left( \delta^2\left(\sum_{j=0}^{\lfloor nt_2 \rfloor -1}\left< Df^{(3)}(\hat{X}_{\frac jn}), \partial_{\frac jn}^{\otimes 3}\right>_\hten \right)\right)^2\right] \le C\lfloor nt_2 \rfloor n^{-\frac 43};\eeq
\beq{L3e2}{\mathbb E}\left[\left( \delta\left(\sum_{j=0}^{\lfloor nt_2 \rfloor -1}\left< D^2f^{(3)}(\hat{X}_{\frac jn}), \partial_{\frac jn}^{\otimes 3}\right>_{\hten^{\otimes 2}} \right)\right)^2\right] \le C\lfloor nt_2 \rfloor n^{-\frac 15\lambda}; \eeq
and
\beq{L3e3}{\mathbb E}\left[\left( \sum_{j=0}^{\lfloor nt_2 \rfloor -1}\left< D^3f^{(3)}(\hat{X}_{\frac jn}), \partial_{\frac jn}^{\otimes 3}\right>_\hten \right)^2\right] \le C\lfloor nt_2 \rfloor n^{-2}.\eeq

\noindent{\em Proof of \req{L3e1}}.  By \req{intmult} we have
\begin{align*}
&{\mathbb E}\left[\left( \delta^2\left(\sum_{j=\lfloor nt_1 \rfloor}^{\lfloor nt_2 \rfloor -1}\left< Df^{(3)}(\hat{X}_{\frac jn}), \partial_{\frac jn}^{\otimes 3}\right>_\hten \right)\right)^2\right] \\
&\qquad \le {\mathbb E}\left[ \sum_{j,k=0}^{\lfloor nt_2 \rfloor -1} \left< \left<Df^{(3)}(\hat{X}_{\frac jn}), \partial_{\frac jn}^{\otimes 3}\right>_\hten ,
\left<Df^{(3)}(\hat{X}_{\frac kn}), \partial_{\frac kn}^{\otimes 3}\right>_\hten \right>_{\hten^{\otimes 2}}\right]\\
&\qquad + 2{\mathbb E}\left[ \sum_{j,k=0}^{\lfloor nt_2 \rfloor -1} \left< \left<D^2f^{(3)}(\hat{X}_{\frac jn}), \partial_{\frac jn}^{\otimes 3}\right>_{\hten^{\otimes 2}} ,
\left<D^2f^{(3)}(\hat{X}_{\frac kn}), \partial_{\frac kn}^{\otimes 3}\right>_{\hten^{\otimes 2}} \right>_\hten\right]\\
&\qquad + {\mathbb E}\left[ \sum_{j,k=0}^{\lfloor nt_2 \rfloor -1} \left<D^3f^{(3)}(\hat{X}_{\frac jn}), \partial_{\frac jn}^{\otimes 3}\right>_{\hten^{\otimes 3}}\cdot
\left<D^3f^{(3)}(\hat{X}_{\frac kn}), \partial_{\frac kn}^{\otimes 3}\right>_{\hten^{\otimes 3}}\right]\\
&\le \;  \sup_{0\le j< \lfloor nt_2 \rfloor} {\mathbb E}\left[ f^{(4)}(\hat{X}_{\frac jn})^2\right] 
\sup_{0\le j< \lfloor nt_2 \rfloor}\left< \hat{\varepsilon}_{\frac jn} , \partial_{\frac jn}\right>_\hten^2 \sum_{j,k=0}^{\lfloor nt_2 \rfloor -1}\left< \partial_{\frac jn} , \partial_{\frac kn}\right>_\hten^2 \\
&\quad + \sup_{0\le j< \lfloor nt_2 \rfloor} {\mathbb E}\left[ f^{(5)}(\hat{X}_{\frac jn})^2\right] 
\sup_{0\le j< \lfloor nt_2 \rfloor}\left<\hat{\varepsilon}_{\frac jn} , \partial_{\frac jn}\right>_\hten^4 \sum_{j,k=0}^{\lfloor nt_2 \rfloor -1}\left|\left< \partial_{\frac jn} , \partial_{\frac kn}\right>_\hten\right| \\
&\quad + \sup_{0\le j< \lfloor nt_2 \rfloor} {\mathbb E}\left[ f^{(6)}(\hat{X}_{\frac jn})^2\right] 
\sum_{j,k=0}^{\lfloor nt_2 \rfloor -1}\left|\left< \hat{\varepsilon}_{\frac jn} , \partial_{\frac jn}\right>_\hten^3\right| \left|\left<\hat{\varepsilon}_{\frac kn} , \partial_{\frac kn}\right>_\hten^3\right|.
\end{align*}
By condition (0) and Lemma 3.2.a and 3.2.d,
\begin{align*}
&\sup_{0\le j< \lfloor nt_2 \rfloor} {\mathbb E}\left[ f^{(4)}(\hat{X}_{\frac jn})^2\right] 
\sup_{0\le j< \lfloor nt_2 \rfloor}\left< \hat{\varepsilon}_{\frac jn} , \partial_{\frac jn}\right>_\hten^2 \sum_{j,k=0}^{\lfloor nt_2 \rfloor -1}\left< \partial_{\frac jn} , \partial_{\frac kn}\right>_\hten^2 \le Cn^{-\frac 43}\lfloor nt_2 \rfloor; \; \text{ and}\\
&\sup_{0\le j< \lfloor nt_2 \rfloor} {\mathbb E}\left[ f^{(5)}(\hat{X}_{\frac jn})^2\right] 
\sup_{0\le j< \lfloor nt_2 \rfloor}\left<\hat{\varepsilon}_{\frac jn} , \partial_{\frac jn}\right>_\hten^4 \sum_{j,k=0}^{\lfloor nt_2 \rfloor -1}\left|\left< \partial_{\frac jn} , \partial_{\frac kn}\right>_\hten\right| \le Cn^{-\frac 53}\lfloor nt_2 \rfloor.
\end{align*}
Then by condition (0), Lemma 3.2.a and 3.2.b,
\begin{align*}
&\sup_{0\le j< \lfloor nt_2 \rfloor} {\mathbb E}\left[ f^{(6)}(\hat{X}_{\frac jn})^2\right] 
\sum_{j,k=0}^{\lfloor nt_2 \rfloor -1}\left|\left< \hat{\varepsilon}_{\frac jn} , \partial_{\frac jn}\right>_\hten^3\right| \left|\left<\hat{\varepsilon}_{\frac kn} , \partial_{\frac kn}\right>_\hten^3\right|\\
&\qquad \le C\sup_{0\le j \le\lfloor nt_2 \rfloor} \left<\hat{\varepsilon}_{\frac jn}, \partial_{\frac jn}\right>_\hten^2 \left(\sum_{j=0}^{\lfloor nt_2\rfloor -1} \left<\hat{\varepsilon}_{\frac jn}, \partial_{\frac jn}\right>_\hten^2 \right)^2 \le Cn^{-2}.\end{align*}

\noindent{\em  Proof of \req{L3e2} and \req{L3e3}}.   The same estimates apply for the other terms, since
\begin{align*}
&{\mathbb E}\left[\left( \delta\left(\sum_{j=0}^{\lfloor nt_2 \rfloor -1}\left< D^2f^{(3)}(\hat{X}_{\frac jn}), \partial_{\frac jn}^{\otimes 3}\right>_{\hten^{\otimes 2}} \right)\right)^2\right]\\
&\quad \le \left| {\mathbb E}\left[ \sum_{j,k=0}^{\lfloor nt_2 \rfloor -1} \left< \left<D^2f^{(3)}(\hat{X}_{\frac jn}), \partial_{\frac jn}^{\otimes 3}\right>_{\hten^{\otimes 2}} ,
\left<D^2f^{(3)}(\hat{X}_{\frac kn}), \partial_{\frac kn}^{\otimes 3}\right>_{\hten^{\otimes 2}} \right>_\hten\right]\right|\\
&\qquad + \left|{\mathbb E}\left[ \sum_{j,k=0}^{\lfloor nt_2 \rfloor -1} \left<D^3f^{(3)}(\hat{X}_{\frac jn}), \partial_{\frac jn}^{\otimes 3}\right>_{\hten^{\otimes 3}}\cdot
\left<D^3f^{(3)}(\hat{X}_{\frac kn}), \partial_{\frac kn}^{\otimes 3}\right>_{\hten^{\otimes 3}}\right]\right|
\end{align*}
and \req{L3e3} is bounded in the above computation as well.
\end{proof}

\subsection{Weak convergence of non-trivial part of $3^{rd}$ order term}
We are now ready to apply Theorem 2.3 to the term
$$\sum_{j=0}^{\Nt -1} \delta^3\left( f^{(3)}(\hat{X}_{\frac jn})\partial_{\frac jn}^{\otimes 3}\right).$$
Let $0=t_0 < t_1 < \dots < t_d \le T$ be a finite set of real numbers, and for $i = 1, \dots, d$ define
$$u_n^i = \sum_{j=\lfloor nt_{i-1} \rfloor}^{\lfloor nt_i \rfloor -1} f^{(3)}(\hat{X}_{\frac jn})\partial_{\frac jn}^{\otimes 3};$$
and define the $d-$dimensional vector $F_n = (F_n^1, \dots, F_n^d)$, where each $F_n^i =  \delta^3(u_n^i)$.

To satisfy the CLT conditions, we must deal with terms of the following forms (see Remark 2.4):
\begin{enumerate}
 \item $\left< u_n^i, h \right>_{\hten^{\otimes 3}}$ for $h \in \hten^{\otimes 3}$,
 \item  $\left< u_n^i, DF_n^j \otimes h \right>_{\hten^{\otimes 3}}$ for $h \in \hten^{\otimes 2}$,
 \item $  \left< u_n^i, D^2 F_n^j \otimes h_1 \right>_{\hten^{\otimes 3}}
 +\left<u_n^i, DF_n^j \otimes DF_n^k \otimes h_2\right>_{\hten^{\otimes 3}}$ for $h_1,~h_2 \in \hten$, and
 \item $\left< u_n^i, D^3F_n^i\right>_{\hten^{\otimes 3}} + \left< u_n^i, D^3F_n^j\right>_{\hten^{\otimes 3}}$ $$+ \left< u_n^i, D^2F_n^j \otimes DF_n^k \right>_{\hten^{\otimes 3}} +  \left< u_n^i, DF_n^j \otimes DF_n^k \otimes DF_n^\ell \right>_{\hten^{\otimes 3}}. $$
 \end{enumerate}
We must show that all terms converge to zero {\em except} for the terms
$\left< u_n^i, D^3F_n^i\right>_{\hten^{\otimes 3}}$, $i = 1, \dots, d$, which will converge stably to a Gaussian random vector (Lemma 3.11).

\begin{lemma}For each $i,j,k,\ell = 1, \dots, d$, the following terms vanish in $L^1(\Omega)$ as $n \to \infty$:
	\begin{enumerate}[(a)]
	\item$\left< u_n^i, h\right>_{\hten^{\otimes 3}}$ for each $h \in \hten^{\otimes 3}$.
	\item$\left< u_n^i, DF_n^j\otimes h\right>_{\hten^{\otimes 3}}$ for each $h \in \hten^{\otimes 2}$.
	\item$\left< u_n^i, D^2F_n^j\otimes h\right>_{\hten^{\otimes 3}} + \left< u_n^i, DF_n^j\otimes DF_n^k \otimes h\right>_{\hten^{\otimes 3}}$ for $h \in \hten$.\end{enumerate}\end{lemma}

\begin{proof}
We begin with two estimates that will be needed.  For each $1 \le i \le d$,
\beq{Norm_DF} {\mathbb E}\left\| DF^i_n \right\|^2_\hten < C; \; \text{ and} \eeq
\beq{Norm_D2F} {\mathbb E}\left\| D^2F^i_n \right\|^2_{\hten^{\otimes 2}} < C. \eeq

\medskip
\noindent{\em Proof of \req{Norm_DF}}.  Let $a_i = \lfloor nt_{i-1} \rfloor$ and $b_i = \lfloor nt_i \rfloor$.  By Lemma 2.1.b,
\[ DF_n^i = \delta^3(Du_n^i) + 3\delta^2(u_n^i).\]
Hence, using Lemma 2.1.c,
\begin{align*}{\mathbb E}\left\| DF^i_n \right\|^2_\hten &\le 2\sum_{j,k=a_i}^{b_i -1} {\mathbb E}\left[ \delta^3\left( f^{(4)}(\Xhat)\partial_{\frac jn}^{\otimes 3}\right)\delta^3\left( f^{(4)}(\hat{X}_{\frac kn})\partial_{\frac kn}^{\otimes 3}\right)\right]\left< \hat{\varepsilon}_{\frac jn}, \hat{\varepsilon}_{\frac kn}\right>_\hten\\
&\quad + 18\sum_{j,k=a_i}^{b_i -1} {\mathbb E}\left[ \delta^2\left( f^{(3)}(\Xhat)\partial_{\frac jn}^{\otimes 2}\right)\delta^2\left( f^{(3)}(\hat{X}_{\frac kn})\partial_{\frac kn}^{\otimes 2}\right)\right]\left< \partial_{\frac jn}, \partial_{\frac kn}\right>_\hten\\
&=2\sum_{j,k=a_i}^{b_i -1}\sum_{\ell =0}^3 {\binom 3\ell}^2 ~{\mathbb E}\left[ f^{(7-\ell)}(\Xhat)f^{(7-\ell)}(\hat{X}_{\frac kn})\right]\left< \partial_{\frac jn}, \partial_{\frac kn}\right>_\hten^3\left< \hat{\varepsilon}_{\frac jn}, \hat{\varepsilon}_{\frac kn}\right>_\hten^\ell\\
&\quad + 18\sum_{j,k=a_i}^{b_i -1}\sum_{\ell =0}^2 {\binom 2\ell}^2 ~ {\mathbb E}\left[ f^{(5-\ell)}(\Xhat)f^{(5-\ell)}(\hat{X}_{\frac kn})\right]\left< \partial_{\frac jn}, \partial_{\frac kn}\right>_\hten^3\left< \hat{\varepsilon}_{\frac jn}, \hat{\varepsilon}_{\frac kn}\right>_\hten^\ell\\
&\le C
\end{align*}
by condition (0) and Lemma 3.2.d.  The proof of \req{Norm_D2F} follows the same lines, using Lemma 2.1.b to obtain
\[D^2F_n^i = \delta^3(D^2u_n^i) + 6\delta^2(Du_n^i) + 6\delta(u_n^i).\]

Now for the main proof.  Without loss of generality, we may assume that each $h\in \hten$ is of the form $\varepsilon_\tau$ for some $0\le \tau\le T$ (see Remark 2.5).  Then for (a) we have:
\begin{align*}
{\mathbb E}\left| \left<u_n^i, h\right>_{\hten^{\otimes 3}}\right| &= {\mathbb E}\left|\sum_{m = a_i}^{b_i -1}
\left< f^{(3)}(\hat{X}_{\frac mn})\partial_{\frac mn}^{\otimes 3}, \varepsilon_\tau \otimes \varepsilon_u \otimes \varepsilon_v\right>_{\hten^{\otimes 3}}\right| \\
&\le \sup_{a_i\le m\le b_i}\left\{{\mathbb E}\left| f^{(3)}(\hat{X}_{\frac mn})\right|\right\} \sup_{\tau, m} \left\{\left< \varepsilon_\tau, \partial_{\frac mn}\right>_\hten^2\right\} \sum_{m=a_i}^{b_i-1} \left|\left< \varepsilon_\tau, \partial_{\frac mn}\right>_\hten\right|\\
&\le Cn^{-2\lambda},\end{align*}
where we used Lemma 3.2.a and Lemma 3.2.c.  For (b),
\begin{align*}
{\mathbb E}\left| \left< u_n^i, DF_n^j \otimes \varepsilon_\tau \otimes \varepsilon_u\right>_{\hten^{\otimes 3}}\right|
&\le \sqrt{\sup_{a_i\le m\le b_i}{\mathbb E}\left|f^{(3)}(\hat{X}_{\frac mn})\right|^2}\sum_{m=a_i}^{b_i-1} \left({\mathbb E}\left< \partial_{\frac mn}^{\otimes 3}, DF_n^j \otimes \varepsilon_\tau \otimes \varepsilon_u\right>_{\hten^{\otimes 3}}^2\right)^{\frac 12}\\
&\le C\sup_m \left\| \partial_{\frac mn}\right\|_\hten~ \sqrt{{\mathbb E}\| DF_n^j\|_\hten^2}~  \sum_{m=a_i}^{b_i -1} \left|\left< \partial_{\frac mn}, \varepsilon_\tau\right>_\hten\left< \partial_{\frac mn}, \varepsilon_u\right>_\hten\right|.
\end{align*}
By condition (i), $\| \partial_{\frac mn}\|_\hten \le Cn^{-\frac 16}$, and so by \req{Norm_DF} and Lemma 3.2.c we have an upper bound of $Cn^{-\frac 16-2\lambda}$.  For (c), by similar reasoning along with Lemma 3.2.c and \req{Norm_D2F},
\begin{align*}
{\mathbb E}\left| \left< u_n^i, D^2F_n^j \otimes \varepsilon_\tau\right>_{\hten^{\otimes 3}}\right| 
&\le \sqrt{\sup_{a_i\le m\le b_i}{\mathbb E}\left|f^{(3)}(\hat{X}_{\frac mn})\right|^2} \sum_{m=a_i}^{b_i -1}\left( {\mathbb E}\left< \partial_{\frac mn}^{\otimes 3} , D^2F_n^j \otimes \varepsilon_\tau\right>_{\hten^{\otimes 3}}^2\right)^{\frac 12}\\
&\le C\sup_m \left\| \partial_{\frac mn}\right\|_{\hten} \sqrt{{\mathbb E}\| D^2F_n^j\|_{\hten^{\otimes 2}}^2}\sum_{m=a_i}^{b_i -1} \left|\left< \partial_{\frac mn}, \varepsilon_\tau\right>_\hten\right|\\
&\le Cn^{-\frac 16}.\end{align*}
The estimate is similar for the term ${\mathbb E}\left| \left< u_n^i, DF_n^j \otimes DF_n^k \otimes \varepsilon_t\right>_{\hten^{\otimes 3}}\right|$.
\end{proof}

Now we focus on the terms $\left< u_n^i, D^3F_n^j\right>_{\hten^{\otimes 3}}.$  By Lemma 2.1.b,
$$D^3F_n^j = \delta^3(D^3u_n^j) + 9\delta^2(D^2u_n^j) + 18\delta(Du_n^j) + 6u_n^j,$$
so that $\left< u_n^i, D^3F_n^j\right>_{\hten^{\otimes 3}}$ can be written as,
\begin{align*}
&\sum_{\ell,m}f^{(3)}(\hat{X}_{\frac mn})\delta^3\left(f^{(6)}(\hat{X}_{\frac{\ell}{n}})\partial_{\frac{\ell}{n}}^{\otimes 3}\right)\left<\partial_{\frac mn}, \hat{\varepsilon}_{\frac{\ell}{n}}\right>_\hten^3\\
&\quad +9\sum_{\ell,m}f^{(3)}(\hat{X}_{\frac mn})\delta^2\left(f^{(5)}(\hat{X}_{\frac{\ell}{n}})\partial_{\frac{\ell}{n}}^{\otimes 3}\right)\left<\partial_{\frac mn}, \hat{\varepsilon}_{\frac{\ell}{n}}\right>_\hten^2\left<\partial_{\frac mn}, \partial_{\frac{\ell}{n}}\right>_\hten\\
&\quad +18\sum_{\ell,m}f^{(3)}(\hat{X}_{\frac mn})\delta\left(f^{(4)}(\hat{X}_{\frac{\ell}{n}})\partial_{\frac{\ell}{n}}\right)\left<\partial_{\frac mn}, \hat{\varepsilon}_{\frac{\ell}{n}}\right>_\hten\left<\partial_{\frac mn}, \partial_{\frac{\ell}{n}}\right>_\hten^2 + 6\left< u_n^i, u_n^j\right>_{\hten^{\otimes 3}}\\
&:= A_n(i,j) + B_n(i,j) + C_n(i,j) + 6\left< u_n^i, u_n^j\right>_{\hten^{\otimes 3}},
\end{align*}
where $m, \ell$ are the indices for $u_n^i, D^3F_n^j$, respectively, with $\lfloor nt_{i-1} \rfloor \le m \le \lfloor nt_i \rfloor$, and $\lfloor nt_{j-1} \rfloor \le \ell \le \lfloor nt_j \rfloor$.

\begin{lemma}  For each $1 \le i,j \le d$ we have
\[ \left|\left< u_n^i, D^3F_n^j\right>_{\hten^{\otimes 3}} ~-~ 6~\delta_{ij}  \left< u_n^i, u_n^j\right>_{\hten^{\otimes 3}}\right| \stackrel{{\mathbb P}}{\longrightarrow} 0\]
as $n \to \infty$, where $\delta_{ij}$ is the Kronecker delta.\end{lemma}

\begin{proof}
We will show that for each $1 \le i,j \le d$
$$\lim_{n \to \infty} {\mathbb E}\,\left|A_n(i,j)\right| = \lim_{n \to \infty} {\mathbb E}\,\left|B_n(i,j)\right|=\lim_{n \to \infty} {\mathbb E}\,\left|C_n(i,j)\right|=0.$$
and moreover, if  $i \neq j$ then
$\lim_{n \to \infty} {\mathbb E}\left|\left< u_n^i, u_n^j \right>_{\hten^{\otimes 3}}\right| =0.$

To begin with, observe that if $g(x)$ is a function satisfying condition (0), then it follows from condition (i) and Lemma 2.1.c that that for $q=1, 2, 3$,
\beq{partial_norm}  \sup_j \left\| \delta^q\left(g(\Xhat)\partial_{\frac jn}^{\otimes q}\right) \right\|_{L^2} \le C\sup_j \left\| \partial_{\frac jn}\right\|^q_\hten \le  Cn^{-\frac q6}.\eeq
For the terms $A_n(i,j), B_n(i,j), C_n(i,j)$ we include the case $i=j$.
We have
\begin{align*}
{\mathbb E}\left| A_n(i,j)\right| &\le \sup_m \left\| f^{(3)}(\hat{X}_{\frac mn})\right\|_{L^2}~\sup_\ell\left\| \delta^3 \left(f^{(6)}(\hat{X}_{\frac{\ell}{n}})\partial_{\frac{\ell}{n}}^{\otimes 3}\right)\right\|_{L^2}~\sup_{\ell, m} \left|\left< \partial_{\frac mn},\hat{\varepsilon}_{\frac{\ell}{n}}\right>_\hten\right| \sum_{\ell, m}\left< \partial_{\frac mn},  \hat{\varepsilon}_{\frac{\ell}{n}}\right>_\hten^2.
\end{align*}
Using condition (0), \req{partial_norm}, and Lemma 3.2.a, respectively, we have
\[{\mathbb E}\left| A_n(i,j)\right| \le Cn^{-\frac 12 -\lambda} \sum_{\ell, m}\left< \partial_{\frac mn},  \hat{\varepsilon}_{\frac{\ell}{n}}\right>_\hten^2,\]
so that Lemma 3.2.c gives ${\mathbb E}\left| A_n(i,j)\right| \le Cn^{-\frac 12-3\lambda}$. 
   Next, using condition (0), \req{partial_norm} and Lemma 3.2.a, 
\begin{align*}
{\mathbb E}\left| B_n(i,j)\right| &\le 9\sup_m \left\| f^{(3)}(\hat{X}_{\frac mn})\right\|_{L^2}~\sup_\ell\left\| \delta^2 \left(f^{(5)}(\hat{X}_{\frac{\ell}{n}})\partial_{\frac{\ell}{n}}^{\otimes 2}\right)\right\|_{L^2}~\sup_{\ell, m} \left< \partial_{\frac mn}, \hat{\varepsilon}_{\frac{\ell}{n}}\right>_\hten^2 \sum_{\ell, m}\left|\left< \partial_{\frac mn}, \partial_{\frac{\ell}{n}}\right>_\hten\right|\\
&\le Cn^{-\frac 13 -2\lambda} \sum_{\ell, m}\left| \left<\partial_{\frac{\ell}{n}}, \partial_{\frac mn}\right>_\hten\right|; \end{align*}
and so by Lemma 3.2.d,
\[{\mathbb E}\left| B_n(i,j)\right| \le Cn^{-\frac 23-2\lambda} \max\{ \lfloor nt_i\rfloor, \lfloor nt_j \rfloor\},\]
which converges to zero since $2\lambda > 1/3$.  Similarly for $C_n(i,j)$ using Lemma 3.2.d, 
\begin{align*}
{\mathbb E}\left| C_n(i,j)\right| &\le 18\sup_m \left\| f^{(3)}(\hat{X}_{\frac mn})\right\|_{L^2}~\sup_\ell\left\| \delta \left(f^{(4)}(\hat{X}_{\frac{\ell}{n}})\partial_{\frac{\ell}{n}}\right)\right\|_{L^2}~\sup_{\ell, m} \left|\left< \partial_{\frac mn}, \hat{\varepsilon}_{\frac{\ell}{n}}\right>_\hten\right| \sum_{\ell, m}\left|\left< \partial_{\frac mn}, \partial_{\frac{\ell}{n}}\right>_\hten^2\right|\\
&\le Cn^{-\frac 16-\lambda}\sum_{\ell, m} \left<\partial_{\frac{\ell}{n}}, \partial_{\frac mn}\right>_\hten^2 \le Cn^{-\frac 56-\lambda}\max\{ \lfloor nt_i\rfloor, \lfloor nt_j \rfloor\} \le Cn^{-\lambda + \frac 16}.
\end{align*}

For the second part, we may assume $i < j$.  Using Lemma 3.2.e,
\begin{align*}
{\mathbb E}\left| \left< u_n^i, u_n^j \right>_{\hten^{\otimes 3}}\right| & \le  \sup_m \left\| f^{(3)}(\hat{X}_{\frac mn})\right\|_{L^2}^2~\sum_{\ell, m}\left| \left< \partial_{\frac mn}, \partial_{\frac{\ell}{n}}\right>_\hten^3\right|\\
&\le Cn^{-1}\lfloor nt_j \rfloor^{\epsilon},
\end{align*}
which converges to zero because $\epsilon < 1$.
\end{proof}

\begin{lemma}  Using the above notation, for each $1 \le i,j,k,l\le d$ we have
\beq{D2F}\lim_{n \to \infty}{\mathbb E}\left[ \left< u_n^i, D^2F_n^j \otimes DF_n^k \right>_{\hten^{\otimes 3}}^2\right] =0, \text{ and}\eeq
\beq{DFDFDF}\lim_{n \to \infty}{\mathbb E}\left[ \left< u_n^i, DF_n^j \otimes DF_n^k \otimes DF_n^\ell \right>_{\hten^{\otimes 3}}^2\right] =0.\eeq
\end{lemma}

The proof of this lemma is deferred to Section 5.

\medskip
Lemmas 3.8, 3.9 and 3.10 show that condition (a) of Theorem 2.3 is satisfied, and moreover that the only non-trivial terms are of the form $6\left< u_n^i, u_n^i \right>_{\hten^{\otimes 3}}$.  It remains to establish the convergence of these terms to a non-negative random variable $6s_i^2$.  With this result, it follows from Theorem 2.3 that the couple $(X, F_n)$ converges stably to $(X, \zeta)$, where $\zeta = (\zeta_1, \dots, \zeta_d)$ is a vector whose components are conditionally independent Gaussian random variables with mean zero and variance $6s_i^2$.

\begin{lemma}  For each $1 \le i \le d$, conditioned on $X$,
$$\lim_{n\to \infty} \left< u_n^i, u_n^i \right>_{\hten^{\otimes 3}} = s_i^2,$$
where each $s_i^2$ has the form
$$s_i^2 = s(t_i)^2 - s(t_{i-1})^2 = \int_{t_{i-1}}^{t_i} f^{(3)}(X_s)^2 \eta(ds).$$
It follows that on the subinterval $[t_{i-1}, t_i]$ we have the conditional result
\[ 6\left< u_n^i, u_n^i \right>_{\hten^{\otimes 3}} \longrightarrow 6\int_{t_{i-1}}^{t_i} f^{(3)}(X_s)^2 \eta(ds),\]
almost surely as $n \to \infty$, which implies
\beq{ununconv}F_n^i \law \sqrt{6}\int_{t_{i-1}}^{t_i} f^{(3)}(X_s) dB_s,\eeq
where $\{ B_t, t\ge 0\}$ is a Brownian motion, independent of $X$, with variance $\eta(t)$.
\end{lemma}
\begin{proof}
Let $a = \lfloor nt_{i-1} \rfloor$ and $b=\lfloor nt_i \rfloor$, and recall $\beta_n(j,k) = \left< \partial_{\frac jn}, \partial_{\frac kn}\right>_\hten$, from condition (vi).  We have
$$\left< u_n^i, u_n^i \right>_{\hten^{\otimes 3}} = \sum_{j,k=a}^{b-1} f^{(3)}(\Xhat)f^{(3)}(\hat{X}_{\frac kn}) \beta_n(j,k)^3.$$
For each $n$, define a discrete measure on $\{ 1, 2, \dots \}^{\otimes 2}$ by
$$\mu_n := \sum_{j,k=0}^\infty \beta_n(j,k)^3 \delta_{jk},$$
where $\delta_{jk}$ denotes the Kronecker delta.  It follows from condition (vi) that for each $t >0$,
\begin{align*}
\mu\left([0,t]^2\right) &:= \lim_{n \to \infty} \mu_n\left( \Nt , \Nt \right)\\
& = \lim_n \sum_{j,k=0}^{\Nt -1} \beta_n(j,k)^3 = \eta(t).\end{align*}
Moreover, if $0 < s < t$ then
$$\mu_n\left( \lfloor ns \rfloor, \Nt\right) = \mu_n\left( \lfloor ns \rfloor, \lfloor ns \rfloor\right)+\sum_{j=0}^{\lfloor ns \rfloor -1}~\sum_{k= \lfloor ns \rfloor}^{\Nt -1} \beta_n(j,k)^3,$$
which converges to zero because the disjoint sum vanishes by Lemma 3.2.e.  Hence we can conclude that $\mu_n$ converges weakly to the measure given by $\mu ([0,s]\times [0,t]) = \eta (s \wedge t)$.  It follows by continuity of $f^{(3)}(X_t)$ and Portmanteau theorem that
$$ \sum_{j,k=0}^{\Nt -1} f^{(3)}(\Xhat) f^{(3)}(\hat{X}_{\frac kn})  \beta_n(j,k)^3 
=\int_{{\mathbb R}^2} f^{(3)}(X_s) f^{(3)}(X_u) {\mathbf 1}_{s<t} {\mathbf 1}_{u<t} ~\mu_n(ds,du)$$
converges in $L^1(\Omega, \hten)$ to
$$\int_0^t f^{(3)}(X_s)^2 \eta(ds).$$ 
It follows that on the subinterval $[t_{i-1}, t_i]$ we have the result
$$\left< u_n^i , u_n^i \right>_{\hten^{\otimes 3}} \longrightarrow \int_{t_{i-1}}^{t_i} f^{(3)}(X_s)^2 \eta(ds)$$
in $L^1(\Omega, \hten)$ as $n \to \infty$.  Using the It\^o isometry for the above integral, we conclude \req{ununconv}.
\end{proof}

\subsection{Relative compactness of the sequence $F_n(t) + \Delta_n(t)$}
Let
$$G_n(t) := \frac 12 \sum_{j=0}^{\Nt -1}\left( f'(X_{\frac{j+1}{n}}) + f'(X_{\frac jn})\right)\Delta X_{\frac jn}.$$
To establish convergence of $G_n(t)$ in ${\mathbf D}[0,\infty)$, we need to show that $\{ G_n(t)\}$ is relatively compact.  For this, it is enough to show that there exist real numbers $\alpha > 0$, $\beta > 1$ such that for each $T > 0$ and any $0 \le t_1 < t < t_2\le T$ we have,
$${\mathbb E}\left[ \left| G_n(t) - G_n(t_1)\right|^\alpha \left| G_n(t_2) - G_n(t)\right|^\alpha\right] \le C\left( \frac{\lfloor nt_2\rfloor - \lfloor nt_1 \rfloor}{n}\right)^\beta.$$
We will do this in several parts.  From the preceding section we have,
\begin{align*}
G_n(t) &= f(X_{\frac{\Nt}{n}}) - f(X_0) + \frac{1}{12}\sum_{j=0}^{\Nt -1} \delta^3\left( f^{(3)}(\Xhat) \partial^{\otimes 3}_{\frac jn}\right) + \frac{4}{2^55!}\sum_{j=0}^{\Nt -1} f^{(5)}(\Xhat)\Delta X_{\frac jn}^5\\
&\quad - Z_n(t) +\frac{3}{12} \sum_{j=0}^{\Nt -1}\left\| \Delta X_{\frac jn}\right\|^2_{L^2} f^{(3)}(\Xhat)\delta(\partial_{\frac jn}) + P_n(t)\\
&= f(X_{\frac{\Nt}{n}}) - f(X_0) + \frac{1}{12}\sum_{j=0}^{\Nt -1} \delta^3\left( f^{(3)}(\Xhat) \partial^{\otimes 3}_{\frac jn}\right) +\Delta_n(t).\end{align*}
By Lemmas 3.3, 3.4, and 3.7 we have
\begin{align*}
{\mathbb E}\left[ \left( \sum_{j=\lfloor nt_1 \rfloor}^{\lfloor nt_2 \rfloor -1} f^{(5)}(\Xhat)\Delta X_{\frac jn}^5\right)^2 \right] &\le Cn^{-\frac 43} \left(\lfloor nt_2\rfloor - \lfloor nt_1\rfloor\right);\\
 {\mathbb E}\left[ \left(Z_n(t_2) - Z_n(t_1)\right)^2\right] &\le Cn^{-\frac 73}\left( \lfloor nt_2 \rfloor - \lfloor nt_1 \rfloor\right)^2;\text{ and }\\
{\mathbb E}\left[P_n(t)^2\right]&\le Cn^{-\frac 43}\left( \lfloor nt_2 \rfloor - \lfloor nt_1 \rfloor\right).\end{align*}
Each of these estimates has the form
$${\mathbb E}\left[ \left(U_n(t_2) - U_n(t_1)\right)^2\right] \le Cn^{-\beta}\left(\lfloor nt_2 \rfloor -\lfloor nt_1 \rfloor\right)^{\zeta} \le C\left(\frac{\lfloor nt_2 \rfloor -\lfloor nt_1 \rfloor}{n}\right)^{\beta},$$
where $\zeta < \beta$ and $\beta >1$, hence it follows by Cauchy-Schwarz that for $t_1 < t < t_2$ we have
$${\mathbb E}\left[ |U_n(t)-U_n(t_1)|~|U_n(t_2) - U_n(t)|\right] \le C\left(\frac{\lfloor nt_2 \rfloor- \lfloor nt_1 \rfloor}{n}\right)^{\beta},$$
so each of these individual sequences is relatively compact.  For the term,
$$Y_n(t) = \sum_{j=0}^{\Nt -1}\left\| \Delta X_{\frac jn}\right\|^2_{L^2} f^{(3)}(\Xhat)\delta(\partial_{\frac jn}),$$ we have by Lemma 3.6 that $Y_n(t)$ vanishes in probability.  However, to show relative compactness we need a different estimate.

\begin{lemma}For $0 \le t_1 < t_2 \le T$ such that $\lfloor nt_2 \rfloor- \lfloor nt_1 \rfloor\ge 1$, we have
$${\mathbb E}\left[ \left(Y_n(t_2)-Y_n(t_1)\right)^4\right] \le Cn^{-2}\left(\lfloor nt_2 \rfloor- \lfloor nt_1 \rfloor\right)^2 + 
Cn^{-\frac 43-4\lambda}\left(\lfloor nt_2 \rfloor- \lfloor nt_1 \rfloor\right)^{\frac 43+4\lambda}.$$It follows that the sequence $\left\{ Y_n(t)\right\}$ is relatively compact.\end{lemma}
\begin{proof}
Let $\Phi_n := \Phi_n(j_1,j_2,j_3,j_4) =\prod_{i=1}^4 f^{(3)}(\hat{X}_{\frac{j_i}{n}})$, and let $a = \lfloor nt_1\rfloor,$ $b = \lfloor nt_2\rfloor$.  We have
\begin{align*}
{\mathbb E}\left[ \left( \sum_{j=a}^{b-1} \left\|\Delta X_{\frac jn}\right\|^2_{L^2} f^{(3)}(\Xhat) \delta(\partial_{\frac jn})\right)^4\right]
& \le \sup_{a \le j < b} \left\|\Delta X_{\frac jn}\right\|^8_{L^2}\sum_{j_1,j_2,j_3,j_3}\left|{\mathbb E}\left[ \Phi_n(j_1,j_2,j_3,j_4) \delta(\partial_{\frac{j_1}{n}})\delta(\partial_{\frac{j_2}{n}})\delta(\partial_{\frac{j_3}{n}})\delta(\partial_{\frac{j_4}{n}})\right]\right|\\
& \le Cn^{-\frac 43}\sum_{j_1,j_2,j_3,j_4}\left|{\mathbb E}\left[ \left<D\left[\Phi_n \delta(\partial_{\frac{j_1}{n}})\delta(\partial_{\frac{j_2}{n}})\delta(\partial_{\frac{j_3}{n}})\right],\partial_{\frac{j_4}{n}}\right>_\hten\right]\right|\\
&\le Cn^{-\frac 43}\sum_{j_1,j_2,j_3,j_4} \sum_{r=1}^4 \left|{\mathbb E}\left[\Phi_n^{(r)} \delta(\partial_{\frac{j_1}{n}})\delta(\partial_{\frac{j_2}{n}})\delta(\partial_{\frac{j_3}{n}})\right]\left< \hat{\varepsilon}_{\frac{j_r}{n}},\partial_{\frac{j_4}{n}}\right>_\hten\right|\\
&\quad + 3 Cn^{-\frac 43}\sum_{j_1,j_2,j_3,j_4}\left|{\mathbb E}\left[\Phi_n \delta(\partial_{\frac{j_1}{n}})\delta(\partial_{\frac{j_2}{n}})\right]\left< \partial_{\frac{j_3}{n}}, \partial_{\frac{j_4}{n}}\right>_\hten\right|\\
&= Cn^{-\frac 43}\sum_{j_1,j_2,j_3,j_4} \sum_{r=1}^4 \left|{\mathbb E}\left[\left< D\left[\Phi_n^{(r)}\delta(\partial_{\frac{j_1}{n}})\delta(\partial_{\frac{j_2}{n}})\right], \partial_{\frac{j_3}{n}}\right>_\hten\right]\left< \hat{\varepsilon}_{\frac{j_r}{n}},\partial_{\frac{j_4}{n}}\right>_\hten\right|\\
&\quad + 3 Cn^{-\frac 43}\sum_{j_1,j_2,j_3,j_4}\left|{\mathbb E}\left[\left<D\left[\Phi_n \delta(\partial_{\frac{j_1}{n}})\right],\partial_{\frac{j_2}{n}}\right>_\hten\right]\left< \partial_{\frac{j_3}{n}}, \partial_{\frac{j_4}{n}}\right>_\hten\right|\\
\end{align*}
where 
\[ \Phi_n^{(r)} = f^{(4)}(\hat{X}_{\frac{j_r}{n}})\prod_{\stackrel{i=1}{i\neq r}}^4 f^{(3)}(\hat{X}_{\frac{j_i}{n}}).\]
Continuing this process, we obtain terms of the form:
\begin{align*}
&Cn^{-\frac 43}\sum_{j_1,j_2,j_3,j_4}\left|{\mathbb E}\left[\Phi\right]\left< \partial_{\frac{j_1}{n}}, \partial_{\frac{j_2}{n}}\right>_\hten\left< \partial_{\frac{j_3}{n}}, \partial_{\frac{j_4}{n}}\right>_\hten\right|,\\
&Cn^{-\frac 43}\sum_{j_1,j_2,j_3,j_4}\left|{\mathbb E}\left[\partial^2\Phi\right]\left< \hat{\varepsilon}_{\frac{j_a}{n}}, \partial_{\frac{j_1}{n}}\right>_\hten\left< \hat{\varepsilon}_{\frac{j_b}{n}}, \partial_{\frac{j_2}{n}}\right>_\hten\left< \partial_{\frac{j_3}{n}}, \partial_{\frac{j_4}{n}}\right>_\hten\right|, \;\text{ and }\\
&Cn^{-\frac 43}\sum_{j_1,j_2,j_3,j_4}\left|{\mathbb E}\left[\partial^4\Phi\right]\left< \hat{\varepsilon}_{\frac{j_a}{n}}, \partial_{\frac{j_1}{n}}\right>_\hten\left< \hat{\varepsilon}_{\frac{j_b}{n}}, \partial_{\frac{j_2}{n}}\right>_\hten\left< \hat{\varepsilon}_{\frac{j_c}{n}}, \partial_{\frac{j_3}{n}}\right>_\hten\left< \hat{\varepsilon}_{\frac{j_d}{n}}, \partial_{\frac{j_4}{n}}\right>_\hten\right|,
\end{align*}
where $\partial^k\Phi$ represents the appropriate $k^{th}$ derivative of $\Phi$. By Lemma 3.2.c and 3.2.d, the sums of each type have, respectively, upper bounds of the form
$$Cn^{-2}(b-a)^{-2} + Cn^{-\frac 53}(b-a) + Cn^{-\frac 43-4\lambda}(b-a)^{4\lambda},$$
hence we conclude that
$${\mathbb E}\left[\left( Y_n(t_2) - Y_n(t_1)\right)^4\right] \le C\left(n^{-2}(\lfloor nt_2\rfloor -\lfloor nt_1\rfloor)^2 + n^{-\frac 43-4\lambda} (\lfloor nt_2\rfloor -\lfloor nt_1\rfloor)^{\frac 43+4\lambda}\right),$$
As for above terms, it follows by Cauchy-Schwarz that
\begin{align*}{\mathbb E}\left[\left|Y_n(t) - Y_n(t_1)\right|^2\left|Y_n(t_2) - Y_n(t)\right|^2\right] 
& \le C\left(n^{-2}(\lfloor nt_2\rfloor -\lfloor nt_1\rfloor)^2 + n^{-\frac 43-4\lambda} (\lfloor nt_2\rfloor -\lfloor nt_1\rfloor)^{\frac 43+4\lambda}\right),\end{align*}
and thus $\{ Y_n(t)\}$ is relatively compact.   
\end{proof}

\noindent{\em Tightness of $F_n$}.  

To conclude the proof of Theorem 3.1, we want to show that the sequence $\{F_n(t)\}$ satisfies the relative compactness condition.
\begin{lemma}For $0 \le t_1 < t_2 \le T$, write
$$F_n(t_2) - F_n(t_1) = \sum_{j=\lfloor nt_1 \rfloor +1}^{\lfloor nt_2 \rfloor -1} \delta^3 \left( f^{(3)}(\Xhat) \partial_{\frac jn}^{\otimes 3}\right)$$
Then given $0\le t_1 < t < t_2 \le T$, there exists a positive constant $C$ such that
\beq{tight} {\mathbb E} \left[ |F_n(t) - F_n(t_1) |^2 |F_n(t_2) -  F_n(t) |^2 \right] \le C\left(\frac{\lfloor nt_2 \rfloor - \lfloor nt_1 \rfloor}{n}\right)^2. \eeq
\end{lemma}
\begin{proof}
We begin with a general claim about the norm of $DF_n$.  Suppose $a, b$ are nonnegative integers. Let
$$g_a = \sum_{j = \lfloor nt_1 \rfloor}^{\lfloor nt_2 \rfloor -1} f^{(a)}(\Xhat)\partial_{\frac jn}^{\otimes 3}.$$
Then we have
\beq{DF_bound}{\mathbb E}\left[ \| D^b g_a \|^4_{\hten^{\otimes 3+b}} \right]  \le C\left(\frac{\lfloor nt_2 \rfloor - \lfloor nt_1 \rfloor}{n}\right)^2.\eeq

\noindent{\em Proof of \req{DF_bound}}.  
For each $b$ we can write
\begin{align*}{\mathbb E}\left[ \left(\| D^b g_a \|^2_{\hten^{\otimes 3+b}}\right)^2\right]
&\quad ={\mathbb E}\left[\left(\sum_{j, k = \lfloor nt_1 \rfloor}^{\lfloor nt_2\rfloor -1} f^{(a+b)}(\Xhat)f^{(a+b)}(\hat{X}_{\frac kn})
\left< \hat{\varepsilon}_{\frac jn}^{\otimes b} , \hat{\varepsilon}_{\frac kn}^{\otimes b} \right>_{\hten^{\otimes b}} \left<\partial_{\frac jn}^{\otimes 3},  \partial_{\frac kn}^{\otimes 3} \right>_{\hten^{\otimes 3}}\right)^2\right]\\
&\quad \le \sup_{\lfloor nt_1 \rfloor \le j< \lfloor nt_2 \rfloor} \left({\mathbb E}\left| f^{(a+b)}(\Xhat) \right|^4\right)^{\frac 12} \left(\sup_{j,k} \left|\left< \hat{\varepsilon}_{\frac jn} , \hat{\varepsilon}_{\frac kn} \right>_\hten \right|^{2b}\right)\left(
\sum_{j,k = \lfloor nt_1 \rfloor}^{\lfloor nt_2 \rfloor -1}\left|\left<\partial_{\frac jn},  \partial_{\frac kn}\right>_\hten\right|^3\right)^2\\
&\le Cn^{-2}\left( \lfloor nt_2 \rfloor - \lfloor nt_1 \rfloor\right)^2,
\end{align*}
by Lemma 3.2.d.  

\noindent{\em Proof of \req{tight}}.  
By the Meyer inequality \req{Meyer} there exists a constant $c_{2,4}$ such that
$${\mathbb E} \left| \left(\delta^3(u_n)\right)^4 \right| \le c_{3,4} \| u_n \|^4_{{\mathbb D}^{3,4}(\hten^{\otimes 3})},$$
where in this case,
$$u_n = \sum_{j=\lfloor nt_1 \rfloor }^{\lfloor nt_2 \rfloor -1} f^{(3)}(\Xhat) \partial_{\frac jn}^{\otimes 3}$$
and
$$\|u_n\|_{{\mathbb D}^{3,4}(\hten^{\otimes 3})}^4 = {\mathbb E}\| u_n \|_{\hten^{\otimes 3}}^4 +{\mathbb E}\| Du_n \|_{\hten^{\otimes 4}}^4+{\mathbb E}\| D^2u_n \|_{\hten^{\otimes 5}}^4+{\mathbb E}\| D^3u_n \|_{\hten^{\otimes 6}}^4.$$ 
From \req{DF_bound} we have ${\mathbb E}\|D^b u_n\|^4_{\hten^{\otimes 3+b}}\le C_bn^{-2}\left(\lfloor nt_2 \rfloor - \lfloor nt_1 \rfloor \right)^2$ for $b=0,1,2,3$.  From this result, given $0 \le t_1 < t < t_2$, it follows from the H\"{o}lder inequality that
\begin{align*}{\mathbb E} \left[ |F_n(t) - F_n(t_1) |^2 |F_n(t_2) - F_n(t)|^2 \right]  &\le C\left(\frac{\lfloor nt_2 \rfloor - \lfloor nt_1 \rfloor}{n}\right)^2.\end{align*}
\end{proof}

\section{Examples of suitable processes}
\subsection{Bifractional Brownian motion}
The bifractional Brownian motion is a generalization of fractional Brownian motion, first introduced by Houdr\'e and Villa \cite{Houdre}.  It is defined as a centered Gaussian process $B^{H,K} = \{B^{H,K}_t, t \ge 0\}$,with covariance given by,
$${\mathbb E} [B_t^{H,K} B^{H,K}_s ] = \frac{1}{2^K}\left( t^{2H} + s^{2H} \right)^K - \frac{1}{2^K}|t-s|^{2HK},$$
where $H \in (0,1),$ $K\in(0,1]$ (Note that the case $K=1$ corresponds to fractional Brownian motion with Hurst parameter $H$).  The reader may refer to \cite{RussoTudor06} and \cite{Lei} for further discussion of properties.

In this section, we show that the results of Section 3 are valid for bifractional Brownian motion with parameter values $H,$ $K$ such that $HK \ge 1/6$.

\begin{proposition}  Let $B_t = \{B^{H,K}_t, t\ge 0\}$ be a bifractional Brownian motion with parameters $H,K$ satisfying $HK = 1/6$.  Then conditions (i) - (v) are satisfied, with $\theta = 2/3$; $\lambda = 1/3$;  
$$\nu = \begin{cases} 5/3&\text{ if }\; H< 1/2\\4H-\frac 13 &\text{ if }\;H\ge 1/2\end{cases};\qquad\text{and }\qquad \gamma = \begin{cases} 2/3+2H & \text{ if }\; H \le 1/2 \text{ and }K < 1\\  5/3 &\text{ otherwise }\end{cases}.$$\end{proposition}
\begin{proof}
\noindent{\em Condition (i)}.  
From Prop. 3.1 of \cite{Houdre} we have
\[{\mathbb E}\left[ (B_t - B_{t-s})^2\right] \le Cs^{2HK} = Cs^{\frac 13}.\]

\medskip
\noindent{\em Condition (ii)}. 
By Fundamental Theorem of Calculus, 
$$\left| {\mathbb E}\left[ B_t^2 - B_{t-s}^2\right]\right| =  t^{2HK} - (t-s)^{2HK} = \int_{-s}^0 2HK(t+\xi)^{2HK-1} d\xi \le Cs(t-s)^{-\frac 23}.$$

\medskip
\noindent{\em Condition (iii)}.
\begin{align*}
&{\mathbb E}\left[(B_t - B_{t-s})^2 - (B_{t-s} - B_{t-2s})^2\right]={\mathbb E}\left[ (B_t - B_{t-2s})(B_t - 2B_{t-s} + B_{t-2s})\right]\\
&\quad=t^{2HK} -\frac{2}{2^K}\left[ t^{2H} + (t-s)^{2H}\right]^K + \frac{1}{2^K}\left[ t^{2H} +(t-2s)^{2H}\right]^K\\
&\;\qquad - \frac{1}{2^K}\left[ t^{2H} +(t-2s)^{2H}\right]^K + \frac{2}{2^K}\left[ (t-s)^{2H} +(t-2s)^{2H}\right]^K.
\end{align*}
In absolute value, this is bounded by
\begin{align*}\frac{1}{2^K}&\left| \left[t^{2H} + t^{2H}\right]^K -2\left[t^{2H} + (t-s)^{2H}\right]^K+ \left[t^{2H} + (t-2s)^{2H}\right]^K\right|\\
&\quad+
\frac{1}{2^K}\left| \left[t^{2H} + (t-2s)^{2H}\right]^K -2\left[(t-s)^{2H} + (t-2s)^{2H}\right]^K+ \left[(t-2s)^{2H} + (t-2s)^{2H}\right]^K\right|.\end{align*}
Both terms have the form 
$$2^{-K}\left| g(t) - 2g(t-s) + g(t-2s)\right| \le Cs^2\sup_{x\in [t-2s,t]}|g''(x)|,$$
for $g(x) = (c+x^{2H})^K$.  We show an upper bound for the first term $g''(x)$, with the other one similar.  We have
\begin{align*}
\sup_{x\in [t-2s,t]} |g''(x)| &\le 4H^2K(1-K)\left[t^{2H} + x^{2H}\right]^{K-2}x^{4H-2} +2HK\left|2H-1\right|\left[ t^{2H} + x^{2H}\right]^{K-1}x^{2H-2}.
\end{align*}
For the above values of $H,K$ we have
\[\sup_{x\in [t-2s,t]} 2HK\left|2H-1\right|\left[ t^{2H} + x^{2H}\right]^{K-1}x^{2H-2} \le C(t-2s)^{2HK-2}.\]
For the first term, if $H < 1/2$ then
\begin{align*}
\sup_{x\in [t-2s,t]} 4H^2K(1-K)\left[t^{2H} + x^{2H}\right]^{K-2}x^{4H-2} &\le C(t-2s)^{2HK-4H+4H-2}\\
&= C(t-2s)^{2HK-2}.\end{align*}
On the other hand, if $H \ge 1/2$, then $t \ge 4s$ implies $t \ge 2(t-2s)$, hence
\begin{align*}
\sup_{x\in [t-2s,t]} 4H^2K(1-K)\left[t^{2H} + x^{2H}\right]^{K-2}x^{4H-2} &\le 4H^2K(1-K)3^{K-2}(t-2s)^{2H(K-2)}x^{4H-2}\\
&\le C(t-2s)^{2HK-4H} = C(t-2s)^{\frac 13 - 4H}.\end{align*}

\medskip
\noindent{\em Condition (iv)}.
First, for the case $|t-r| < 2s$ or $t< 2s$, we have
\begin{align*}
\left|{\mathbb E}\left[ B_r(B_t - B_{t-s})\right]\right| &= \left|{\mathbb E}\left[ B_rB_t - B_rB_{t-s}\right]\right|\\
&\le \frac{1}{2^K}\left( \left[r^{2H} + t^{2H}\right]^K - \left[r^{2H} + (t-s)^{2H}\right]^K\right) + \frac{1}{2^K}\left| |r-t+s|^{2HK} - |r-t|^{2HK}\right|\\
&\le Cs^{2HK} = Cs^{\frac 13}\end{align*}
using the inequality $a^r - b^r \le (a-b)^r$ for $0 < r < 1$.  For $|t-r| \ge 2s$, $t \ge 2s$, we consider two cases. First, assume $r \ge t+2s$. 
\begin{align*}
\left|{\mathbb E}\left[ B_r(B_t - B_{t-s})\right]\right| &= \left|{\mathbb E}\left[ B_rB_t - B_rB_{t-s}\right]\right|\\
&\le \frac{1}{2^K}\left( \left[r^{2H} + t^{2H}\right]^K - \left[r^{2H} + (t-s)^{2H}\right]^K\right) + \frac{1}{2^K}\left| |r-t+s|^{2HK} - |r-t|^{2HK}\right|\\
&=\frac{1}{2^K}\int_{-s}^0 2HK\left[ r^{2H} + (t+\xi)^{2H}\right]^{K-1}(t+\xi)^{2H-1}d\xi + \frac{1}{2^K}\int_0^s 2HK(r-t+\eta)^{2HK-1}d\eta\\
&\le 2^{1-K}HKs(t-s)^{-\frac 23} + 2^{1-K}HKs(r-t)^{-\frac 23},\end{align*}
where we used the fact that $r-t \ge 2s$ implies $r-t \ge 2(r-t-s)$.  On the other hand, if $r \le t-2s$, then the estimate for 
\[\frac{1}{2^K}\left( \left[r^{2H} + t^{2H}\right]^K - \left[r^{2H} + (t-s)^{2H}\right]^K\right)\]
is the same, and for the other term we have,
\begin{align*}\frac{1}{2^K}\left| |r-t+s|^{2HK} - |r-t|^{2HK}\right|
& \le \frac{1}{2^K} \int_{-s}^0 2HK(t-r-\xi)^{2HK-1}d\xi\\ &\le 2^{1-K}HKs(t-r-s)^{2HK-1}\le 2^{\frac 53 -K}HKs(t-r)^{-\frac 23},\end{align*}
hence for either case we have an upper bound of
$Cs\left( (t-s)^{\lambda-1} + |t-r|^{\lambda-1}\right)$ for $\lambda = \frac 13$.  

\medskip
\noindent{\em Condition (v)}.
Assume $t\wedge r \ge 2s$ and $|t-r|\ge 2s$.  We have
\begin{align*}&{\mathbb E}\left[ (B_t - B_{t-s})(B_r - B_{r-s})\right]\\  &\quad= \frac{1}{2^K}\left(\left[t^{2H} + r^{2H}\right]^K - \left[t^{2H}  + (r-s)^{2H}\right]^K \;- \left[(t-s)^{2H} + r^{2H}\right]^K + \left[(t-s)^{2H} + (r-s)^{2H}\right]^K\right) \\
&\qquad\;+\frac{1}{2^K}\left(|t-r+s|^{2HK} - 2 |t-r|^{2HK} + |t-r-s|^{2HK}\right).\end{align*}
This can be interpreted as the sum of a position term, $\frac{1}{2^K}\varphi(t, r, s)$, and a distance term, $\frac{1}{2^K}\psi(t-r, s)$, where
$$\varphi(t, r, s) = \left[t^{2H} + r^{2H}\right]^K - \left[t^{2H}  + (r-s)^{2H}\right]^K- \left[(t-s)^{2H} + r^{2H}\right]^K + \left[(t-s)^{2H} + (r-s)^{2H}\right]^K; \; \text{ and}$$
$$\psi(t-r, s) = |t-r+s|^{2HK} - 2 |t-r|^{2HK} + |t-r-s|^{2HK}.$$
We begin with the position term.  Note that if $K=1$, then $\varphi(t,r,s) = 0$, so we may assume $K < 1$ and $H > \frac{1}{6}$.  Without loss of generality, assume $0< 2s \le r \le t$.  We can write $\varphi(t, r, s)$ as
\begin{align*}&2HK \int_0^s \left(\left[ t^{2H} + (r-\xi)^{2H}\right]^{K-1} (r-\xi)^{2H-1}  -\left[ (t-s)^{2H} + (r-\xi)^{2H}\right]^{K-1} (r-\xi)^{2H-1} \right)d\xi \\
&=\int_0^s \int_0^s 4H^2K(1-K) \left[(t-\eta)^{2H} + (r-\xi)^{2H}\right]^{K-2} (t-\eta)^{2H-1}(r-\xi)^{2H-1}~d\xi ~d\eta, \end{align*}so that
\beq{phi_t-r}|\varphi(t,r,s)| \le 4H^2K(1-K)s^2\left[ (t-s)^{2H} + (r-s)^{2H} \right]^{K-2} (t-s)^{2H-1}(r-s)^{2H-1}.\eeq
Using \req{phi_t-r}, there are 3 cases to consider:
\begin{itemize}
\item If $H < 1/2$, then for $2s \le r \le t-2s$, we have $t-r < t-s$ and \begin{align*}
Cs^2\left[ (t-s)^{2H} + (r-s)^{2H} \right]^{K-2} (t-s)^{2H-1}(r-s)^{2H-1} &\le Cs^2(t-r)^{2HK-2H-1}(r-s)^{2H-1}\\
&=C\left(\frac{s}{r-s}\right)^{1-2H} s^{\frac 13 +\frac 23 +2H}(t-r)^{-\frac 23 -2H}\\ &\le Cs^{\frac 13 +\gamma}|t-r|^{-\gamma},\end{align*}
where $\gamma = \frac 23 + 2H > 1$.
\item  If $H = 1/2$, then $K=1/3$ and for $2s \le r \le t-2s$
\[s^2\left[ (t-s)^{2H} + (r-s)^{2H} \right]^{K-2} (t-s)^{2H-1}(r-s)^{2H-1} \le s^2|t-r|^{-\frac 53}.\]
\item  If $H > 1/2$, then note that for $2s \le r\le t-2s$
\[ s^2\left[ (t-s)^{2H} + (r-s)^{2H} \right]^{K-2} (t-s)^{2H-1}(r-s)^{2H-1} \le s^2(t-s)^{2HK-2} \le s^2|t-r|^{-\frac 53}.\]
\end{itemize}

\medskip
Next, consider the distance term $\psi(t-r, s)$.  Without loss of generality, assume $2s \le r\le t-2s$.
We have 
\begin{align*} 
\left|\psi(t-r,s)\right|&=\left|~ |t-r+s|^{2HK} - 2|t-r|^{2HK} + |t-r-s|^{2HK}~\right| \\
&=\left| \int_0^s \int_{-\xi}^{\xi} 2HK(2HK-1)\left[ t-r+\eta\right]^{2HK-2} d\eta~d\xi \right|\\
&\le Cs^2(t-r-s)^{2HK-2} \le Cs^2|t-r|^{-\frac{5}{3}},\end{align*}
since $|t-r| \ge 2s$ implies $(t-r-s)^{-\frac{5}{3}} \le 2^{\frac{5}{3}}|t-r|^{-\frac{5}{3}}$.  Note that when $K < 1$, then $H< 1/2$ implies $\gamma \le 5/3$, so the upper bound is controlled by $\varphi(t,r,s)$ in this $K=1$ case. 
\end{proof}

\medskip
\begin{proposition}
Let $\{B^{H,K}_t, t\ge 0\}$ be a bifractional Brownian motion with parameters $H \le 1/2$ and $HK = 1/6$.  Then Condition (vi) of Section 4 holds, with the function $\eta(t) = C_Kt$, where
$$C_K = \frac{1}{8^K}\left( 8 + 2\sum_{m=1}^{\infty}\left((m+1)^{\frac 13}-2m^{\frac 13} + (m-1)^{\frac 13}\right)^3\right).$$
\end{proposition}
\begin{proof}
First of all, we write
$$\sum_{j,k=0}^{\Nt -1} \beta_n(j,k)^3 = 2\sum_{j=0}^{\Nt-1} \beta_n(j,0)^3~+~\sum_{j,k=1}^{\Nt-1}\beta_n(j,k)^3.$$
When $j \ge 2$, we have
\begin{align*} |\beta_n(j,0)| &= \left| {\mathbb E}\left[ X_{\frac 1n}(X_{\frac{j+1}{n}} - X_{\frac jn})\right]\right|\\
&\le \frac{1}{2^Kn^{\frac 13}}\left( \left[1+(j+1)^{2H}\right]^K - \left[ 1 + j^{2H}\right]^K\right) + \frac{1}{2^Kn^{\frac 13}}\left| (j-1)^{2HK} - j^{2HK}\right|\\
&\le \frac{1}{2^Kn^{\frac 13}}\int_0^1 2HK\left[ 1+(j+x)^{2H}\right]^{K-1}(j+x)^{2H-1} dx +\frac{1}{2^Kn^{\frac 13}}\int_0^1 2HK(j-1+y)^{2HK-1} dy\\
&\le Cn^{-\frac 13}(j-1)^{-\frac 23}.\end{align*}
Therefore, using Lemma 3.2.a for $\beta_n(0,0)$ and $\beta_n(1,0)$,
\[\sum_{j=0}^{\Nt -1}|\beta_n(j,0)^3| \le 2Cn^{-1} + \sum_{j=2}^{\Nt-1} Cn^{-1}(j-1)^{-2} \le Cn^{-1};\]
and in the rest of the proof we will always assume $j,k \ge 1$.

As in Prop. 4.1, we use the decomposition,
$$\beta_n (j,k) = \frac{1}{2^K}\varphi\left(\frac{j+1}{n}, \frac{k+1}{n}, \frac{1}{n}\right) + \frac{1}{2^K}\psi\left(\frac{j-k}{n}, \frac{1}{n}\right) = 
2^{-K}n^{-\frac 13}\varphi(j+1,k+1,1) +2^{-K}n^{-\frac 13}\psi(j-k,1),$$
which gives
\[ \beta_n(j,k)^3 = \frac{1}{8^Kn}\left( \varphi^3 + 3\varphi^2\psi + 3\varphi\psi^2 +\psi^3\right).\]
 
To begin, we want to show that
\beq{phi_to_0} \lim_{n \to \infty} \sum_{j,k=1}^{\Nt -1} n^{-1} |\varphi(j+1,k+1,1)| ~=~0.\eeq
{\em Proof of \req{phi_to_0}}.  Note that $\varphi = 0$ if $K=1$, so we may assume $K <1$ and $H > 1/6$.  From \req{phi_t-r}, when $t\wedge r \ge 2s$ and $|t-r| \ge 2s$ we have
\begin{align*}
|\varphi(t,r,s)| &\le 4H^2K(1-K)s^2\left[ (t-s)^{2H} + (r-s)^{2H} \right]^{K-2} (t-s)^{2H-1}(r-s)^{2H-1}\\
 &\le Cs^2(t-s)^{HK-1}(r-s)^{HK-1},\end{align*}
so that
\[\left|\varphi(j+1, k+1, 1)\right| \le Cn^{-2HK}j^{HK-1}k^{HK-1}.\]
Recalling the notation $J_d$ from Lemma 3.2.d, we have
\begin{align*}
n^{-1} \sum_{j,k=0}^{\Nt -1} \left|\varphi(j+1, k+1, 1)\right| &=  n^{-1}\sum_{(j,k)\in J_d} \left|\varphi(j+1, k+1, 1)\right| + n^{-1}\sum_{(j,k)\notin J_d} \left|\varphi(j+1, k+1, 1)\right|\\
&\le C\Nt n^{-\frac 43} + Cn^{-\frac 43} \left(\sum_{j=2}^{\Nt -1} j^{HK-1} \right)^2\\
&\le C\Nt n^{\frac 43} + C\Nt^{2HK}n^{-\frac 43} \le Cn^{-\frac 13},\end{align*}
where we used the fact (which follows from Lemma 3.2.a and the definition of $\varphi$ and $\psi$) that $|\varphi(j+1,k+1,1)|$ is bounded. 

%
Hence, \req{phi_to_0} is proved.
It follows from \req{phi_to_0} that
\beq{phipsi_0}
\frac{1}{8^Kn}\sum_{j,k=1}^{\Nt -1} \left| \varphi^3 + 3\varphi^2\psi + 3\varphi\psi^2\right|\to_n 0,\eeq
since $\varphi$ and $\psi$ are both bounded.  Hence, it is enough to consider
\beq{eta_is_psi}\eta(t) = 
\lim_{n\to\infty}\frac{1}{8^K n}\sum_{j,k=0}^{\Nt -1} \psi(j-k,1)^3.\eeq

To evaluate \req{eta_is_psi}, we have
\begin{align*}
&\frac{1}{8^Kn}\sum_{j,k=0}^{\Nt -1} \psi(j-k,1)^3 \\
&\quad = \frac{1}{8^Kn}\sum_{j,k=0}^{\Nt -1}\left(|j-k+1|^{\frac 13} - 2|j-k|^{\frac 13} + |j-k-1|^{\frac 13}\right)^3\\
&\quad = \frac{1}{8^Kn}\sum_{j=0}^{\Nt -1} 2^3 + \frac{2}{8^Kn}\sum_{j=0}^{\Nt -1}\sum_{k=0}^{j-1}\left((j-k+1)^{\frac 13}-2(j-k)^{\frac 13} + (j-k-1)^{\frac 13}\right)^3\\
&\quad = \frac{8\Nt}{8^Kn} + \frac{2}{8^Kn}\sum_{j=1}^{\Nt -1}\sum_{m=1}^{j-1}\left((m+1)^{\frac 13}-2m^{\frac 13} + (m-1)^{\frac 13}\right)^3\\ 
&\quad = \frac{8\Nt}{8^Kn} + \frac{2}{8^Kn}\sum_{j=1}^{\Nt -1}\sum_{m=1}^{\infty}\left((m+1)^{\frac 13}-2m^{\frac 13} + (m-1)^{\frac 13}\right)^3\\
&\qquad - \frac{2}{8^Kn}\sum_{j=1}^{\Nt -1}\sum_{m=j}^{\infty}\left((m+1)^{\frac 13}-2m^{\frac 13} + (m-1)^{\frac 13}\right)^3,\end{align*}
where the last term tends to zero since
$$Cn^{-1}\sum_{j=1}^{\Nt -1}\sum_{m=j}^{\infty}\left((m+1)^{\frac 13}-2m^{\frac 13} + (m-1)^{\frac 13}\right)^3 \le Cn^{-1}\sum_{j=1}^\infty j^{-4} \to_n 0. $$
We therefore conclude that
$\eta(t) = C_K t$,
where
\begin{align*}
C_K 
&= \frac{1}{8^K}\left( 8 + 2\sum_{m=1}^{\infty}\left((m+1)^{\frac 13}-2m^{\frac 13} + (m-1)^{\frac 13}\right)^3\right).\end{align*}
This number is approximately $\frac{7.188}{8^K}$.
\end{proof}

As an immediate consequence of our proof of Theorem 3.1, we have an alternate proof and extension of previous results in Gradinaru {\em et al.}  In \cite{GNRV}, it was proved that \req{gnrv} holds for any fractional Brownian motion with $H > 1/6$, that is, the correction term vanishes.  Following Remark 3.5, we may conclude the following:

\begin{corollary}  Let $B_t = \{B^{H,K}_t, t\ge 0\}$ be a bifractional Brownian motion with parameters $1/6 < HK < 1$.  Then on a fixed interval $[0, T]$ and for $0 < s \le 1$, $B$ satisfies \req{gnrv}. \end{corollary}

\begin{proof}  
Notice that $s \le 1$ implies $s^{2HK} \le s^{\frac 13}$.  With small modifications to the proof of Prop. 4.1, it is easy to verify that conditions (i) - (v) are satisfied when $HK > 1/6$.  In accordance with Remark 3.5, we want to show that
\[\lim_{n\to\infty} \sum_{j,k=0}^{\Nt -1} |\beta_n(j,k)^3| = 0.\]
We may assume $K < 1$.  From Prop. 3.1 of \cite{Houdre}, we have that
\[{\mathbb E}\left[ (B_t - B_{t-s})^2\right] \le Cs^{2HK}.\]
Recalling the notation $J_d$ from Lemma 3.2.d, Cauchy-Schwarz implies for $(j,k) \in J_d$, we have $|\beta_n(j,k)| \le Cn^{-2HK}$.  For $(j,k) \notin J_d$, by \req{phi_t-r} we have
\begin{align*}
\left| \varphi(j+1,k+1,1)\right| &\le 4H^2K(1-K)\left[ j^{2H} + k^{2H}\right]^{K-2}j^{2H-1}k^{2H-1}\\
&\le C|j-k|^{-1-2H(1-K)},\end{align*}
and similar to Prop. 4.1, we have
\[\left| \psi(j-k,1)\right| \le C|j-k|^{2HK-2};\]
hence for $(j,k) \notin J_d$ we have $|\beta_n(j,k)| \le Cn^{-2HK}|j-k|^{-\gamma}$ for $\gamma = \min\{ 1+2H(1-K), 2HK-2\}$.  It follows that
\begin{align*} \left| \sum_{j,k=0}^{\Nt -1} \beta_n(j,k)^3\right|
&\le \sum_{(j,k)\in J_d} |\beta_n(j,k)^3| \,+\sum_{(j,k)\notin J_d} |\beta_n(j,k)^3|\\
&\le \sum_{(j,k)\in J_d} Cn^{-2HK} \,+Cn^{-6HK}\sum_{(j,k)\notin J_d}|j-k|^{-\gamma}\\
&\le Cn^{-6HK}\Nt \end{align*}
so $|\eta(t)| =0$ because $HK > 1/6$.  \end{proof}

\subsection{Extended bifractional Brownian motion}
This process is discussed in a recent paper by Bardina and Es-Sebaiy \cite{Bardina}.  The covariance has the same formula as standard bBm, but it is `extended' in the sense that $1 < K <2$, with $H$ restricted to satisfy $0 < HK <1$.  Within the context of this paper, this allows us to consider values of $1/12 < H < 1/6$.  As in section 4.1, we show computations only for the case $HK=1/6$.  A result similar to Cor. 4.3 can also be shown by modification to the proposition below.

\begin{proposition}  Let $Y = \{Y^{H,K}_t, t\ge 0\}$ be an extended bifractional Brownian motion with parameters $1 < K <2$, $HK = 1/6$.  Then $Y$ satisfies conditions (i) - (vi), with $\theta = 2/3$, $\lambda = (2H) \wedge \frac 13$, and with $\nu$, $\gamma$ and $\eta(t)$ as given in Prop. 4.1. \end{proposition} 

\begin{proof}
Conditions (ii) and (v) are the same as for standard bBm, as shown in Prop. 4.1.  In particular, the decomposition into $\phi(t,r,s)$ and $\psi(t-r,s)$ for condition (v) is the same, so it follows that $\eta(t)$ of condition (vi) has the same form.  The proofs for conditions (i), (iii) and (iv) require some modifications to accept the case $K>1$.

\medskip
\noindent{\em Condition (i)}.  
From Prop. 3 of \cite{Bardina} we have
\[
\left|{\mathbb E}\left[ (Y_t - Y_{t-s})^2\right]\right| \le s^{2HK} = s^{\frac 13}.\]

\medskip
\noindent{\em Condition (iii)}.  First, we have
\begin{align*}
&{\mathbb E}\left[(Y_t - Y_{t-s})^2 - (Y_{t-s} - Y_{t-2s})^2\right]\\
&\qquad=t^{2HK} -\frac{2}{2^K}\left[ t^{2H} + (t-s)^{2H}\right]^K + \frac{2}{2^K}\left[ (t-s)^{2H} +(t-2s)^{2H}\right]^K - \frac{2}{2^K}(t-2s)^{2HK}\\
&\qquad \le 2t^{2HK}-\frac{2}{2^K}\left[ t^{2H} + (t-s)^{2H}\right]^K -2(t-2s)^{2HK} + \frac{2}{2^K}\left[ (t-s)^{2H} +(t-2s)^{2H}\right]^K\\
&\qquad = 4HK\int_{-s}^0 (t+\xi)^{2HK-1} - (t-s+\xi)^{2HK-1}~d\xi\\
&\qquad \le Cs^2(t-2s)^{2HK-2}.  
\end{align*}
On the other hand,
\begin{align*}
&t^{2HK} -\frac{2}{2^K}\left[ t^{2H} + (t-s)^{2H}\right]^K + \frac{2}{2^K}\left[ (t-s)^{2H} +(t-2s)^{2H}\right]^K - \frac{2}{2^K}(t-2s)^{2HK}\\
&\qquad \ge 2(t-s)^{2HK} - \frac{2}{2^K}\left[ t^{2H} +t^{2H}\right]^K + \frac{2}{2^K}\left[ (t-2s)^{2H}+(t-2s)^{2H}\right]^K-2(t-s)^{2HK}\\
&\qquad = -4HK\int_0^s (t-s+\eta)^{2HK-1} - (t-2s+\eta)^{2HK-1}~d\eta\\
&\qquad \ge -Cs^2(t-2s)^{2HK-2},
\end{align*}
hence the term is bounded in absolute value as required, with $\nu = 2-2HK = 5/3$.

\medskip
\noindent{\em Condition (iv)}.  
\begin{align*}
\left|{\mathbb E}\left[ Y_r(Y_t - Y_{t-s})\right]\right| &= \frac{1}{2^K}\left( \left[r^{2H} + t^{2H}\right]^K - \left[r^{2H} + (t-s)^{2H}\right]^K\right) + \frac{1}{2^K}\left( |r-t+s|^{2HK} - |r-t|^{2HK}\right)\\
&\le \frac{1}{2^K}\left( \left[r^{2H} + t^{2H}\right]^K - \left[r^{2H} + (t-s)^{2H}\right]^K\right) + \frac{1}{2^K}\left( (|r-t|+s)^{2HK} - |r-t|^{2HK}\right).
\end{align*}
We consider two cases for the first term.  If $t < 2s$, then by Fundamental Theorem of Calculus,
\begin{align*}
\frac{1}{2^K}\left( \left[r^{2H} + t^{2H}\right]^K - \left[r^{2H} + (t-s)^{2H}\right]^K\right) &\le \frac{1}{2^K}\left(\left[r^{2H} + (2s)^{2H}\right]^K - r^{2HK}\right)\\
&= \frac{K}{2^K}\int_0^{(2s)^{2H}} \left[ r^{2H}+u\right]^{K-1}du \le Cs^{2H}.\end{align*}
If $t \ge 2s$, then
\begin{align*}
\frac{1}{2^K}\left( \left[r^{2H} + t^{2H}\right]^K - \left[r^{2H} + (t-s)^{2H}\right]^K\right) &=2HK\int_{-s}^0\left[ r^{2H} + (t+u)^{2H}\right]^{K-1}(t+u)^{2H-1}du\\
\le CsT^{2H(K-1)}(t-s)^{2H-1} \le Cs(t-s)^{2H-1}.\end{align*}
In particular, if $|r-t| < 2s$ then this is bounded by \[Cs^{2H}\left(\frac{s}{t-s}\right) \le Cs^{2H}.\]
For the second term, if $|r-t| < 2s$, then it easily follows that
\[\frac{1}{2^K}\left( (|r-t|+s)^{2HK} - |r-t|^{2HK}\right) \le Cs^{2HK} \le Cs^{2H};\]
and if $|r-t| \ge 2s$, then by Mean Value
\[
\frac{1}{2^K}\left( (|r-t|+s)^{2HK} - |r-t|^{2HK}\right) \le Cs|r-t|^{2HK-1} \le CsT^{2H(K-1)}|r-t|^{2H-1} \le Cs|r-t|^{2H-1}.\]
In particular, if $t < 2s$ then
\[Cs|r-t|^{2H-1} = Cs^{2H}\left(\frac{s}{|r-t|}\right) \le Cs^{2H}.\]
Hence, we have shown that
\[\left|{\mathbb E}\left[ Y_r(Y_t - Y_{t-s})\right]\right| \le \begin{cases} Cs\left[ (t-s)^{2H-1}+|r-t|^{2H-1}\right]&\text{ if } T\ge 2s ~\text{ and }~|r-t| \ge 2s\\ Cs^{2H} &\text{ otherwise}\end{cases}\]
and so condition (iv) is satisfied by taking $\lambda = \min\{ 2H, \frac 13\}$, where $K \in (1,2)$ implies $\lambda > 1/6$. 
\end{proof}

\bigskip
\subsection{Sub-fractional Brownian motion}
Another variant on fBm is the process known as sub-fractional Brownian motion (sfBm).  This is a centered Gaussian process $\{ Z_t, t\ge 0\}$, with covariance defined by:
\beq{Z_cov} R_h(s,t) = s^h + t^h -\frac 12 \left[ (s+t)^h + |s-t|^h\right],\eeq
with real parameter $h \in (0,2)$.  Some properties of sfBm are given in \cite{Bojdecki} and \cite{Chavez}.  Note that $h=1$ is a standard Brownian motion, and also note the similarity of $R_h(t,s)$ to the covariance of fBm with $H = h/2$. Indeed, in \cite{Chavez} it is shown that sfBm may be decomposed into an fBm with $H = h/2$ and another centered Gaussian process.

Similar to Section 4.1, we discuss only the case $h =1/3$.  For $h > 1/3$, it can be shown that conditions (i)-(vi) are satisfied with $\eta(t) = 0$, hence \req{gnrv} holds.

\begin{proposition}  Let $Z = \{ Z_t, t\ge 0\}$ be a sub-fractional Brownian motion with covariance \req{Z_cov} and parameter $h=1/3$.  Then $Z$ satisfies conditions (i) - (vi) of Section 3; hence Theorem 3.1 holds. For condition (vi) we have
$\eta(t) = C_ht$, where
\[C_h = 1 + \frac 14\sum_{m=1}^\infty\left( (m+1)^{\frac 13} - 2m^{\frac 13} +(m-1)^{\frac 13}\right)^3.\] \end{proposition}

\begin{proof}
\noindent{\em Condition (i)}.  We have
\begin{align*}
{\mathbb E}\left[ (Z_t - Z_{t-s})^2\right] &= R_h(t,t) + R_h(t-s, t-s) - 2R_h(t, t-s)\\
&= -\frac{2^h}{2}t^h + \frac 12 (2t-s)^h -\frac{2^h}{2}(t-s)^h + \frac 12 (2t-s)^h + s^h\\
&= -\frac 12 \left[ (2t)^h - (2t-s)^h \right] - \frac 12 \left[ (2t-2s)^h - (2t-s)^h \right] + s^h.
\end{align*}
This is bounded in absolute value by $Cs^h$, using the inequality $a^h - b^h \le (a-b)^h$.

\medskip
\noindent{\em Condition (ii)}. 
\begin{align*}
\left| {\mathbb E}\left[ Z^2_t - Z^2_{t-s} \right]\right| &= \left|2t^h - \frac{2^h}{2}t^h - 2(t-s)^h + \frac{2^h}{2}(t-s)^h\right|\\
&= \frac{|4-2^h|}{2}\left[ t^h - (t-s)^h\right].\end{align*}
By Mean Value this is bounded by 
\[
Cs(t-s)^{h-1} = Cs(t-s)^{-\frac 23},\]
which implies (ii) with $\theta = 2/3$.

\medskip
\noindent{\em Condition (iii)}.  
\begin{align*}
{\mathbb E}\left[ (Z_t - Z_{t-s})^2 - (Z_{t-s}- Z_{t-2s})^2\right] 
&= R_h(t,t) - 2R_h(t,t-s) +2R_h(t-s, t-2s) - R_h(t-2s,t-2s)\\
&=-\frac{(2t)^h}{2} + (2t-s)^h - (2t-3s)^h + \frac 12(2t-4s)^h\\
&= -\frac 12 \left[ (2t)^h - 2(2t-s)^h + (2t-2s)^h\right] +\frac 12 \left[ (2t-2s)^h - 2(2t-3s)^h + (2t-4s)^h\right].
\end{align*}
By Mean Value, these terms are bounded in absolute value by 
\begin{align*}
Cs^2(2t-4s)^{h-2} &\le Cs^{\frac 13 + \nu}(t-s)^{-\nu}\end{align*}
for $\nu = 5/3$.

\medskip
\noindent{\em Condition (iv)}.  
\begin{align*}
\left|{\mathbb E}\left[ Z_r(Z_t - Z_{t-s})\right]\right| &= \left|R_h(r,t) - R_h(r,t-s)\right|\\
&=\left| t^h - (t-s)^h - \frac 12\left[ (r+t)^h - (r+t-s)^h\right]+ \frac 12 \left( |r-t+s|^h - |r-t|^h\right)\right|\\
\end{align*}
Note that the above expression is always bounded by $Cs^h$ by the inequality $a^h - b^h \le (a-b)^h$.  Hence, the bound is satisfied for the cases $t <2s$ or $|t-r|<2s$.  Assuming $t \ge 2s$, $|r-t| \ge 2s$, we have
\begin{align*}
&\left| t^h - (t-s)^h - \frac 12\left[ (r+t)^h - (r+t-s)^h\right]+ \frac 12 \left( |r-t+s|^h - |r-t|^h\right)\right|\\
&\qquad\quad\le h\int_{-s}^0 (t+u)^{h-1} du + \frac h2 \int_{-s}^0 (r+t+u)^{h-1}du + \frac h2\int_0^s (|r-t|+u)^{h-1} du\\
&\qquad\quad\le Cs(t-s)^{-\frac 23} + Cs(|r-t|-s)^{-\frac 23}.
\end{align*}
For $|r-t|\ge 2s$, we have $(|r-t|-s) \ge \frac 12 |r-t|$, so
\[Cs(t-s)^{-\frac 23} + Cs(|r-t|-s)^{-\frac 23} \le Cs(t-s)^{\lambda-1} + Cs|r-t|^{\lambda-1}\]
for $\lambda = 1/3$. 

\medskip
\noindent{\em Condition (v)}.
\begin{align*}
{\mathbb E}\left[ (Z_t - Z_{t-s})(Z_r - Z_{r-s})\right] &= R_h(t,r) - R_h(t-s,r) - R_h(t, r-s) +R_h(t-s,r-s)\\
&=-\frac 12 \left[(t+r)^h - 2(t+r-s)^h + (t+r-2s)^h\right] + \frac 12 \left[ |t-r+s|^h -2|t-r| + |t-r-s|^h \right].\end{align*}
Assuming that $|t-r| \ge 2s$, by Mean Value this is bounded in absolute value by
\[ Cs^2|t-r-s|^{h-2} \le Cs^2|t-r|^{h-2}\]
since $|t-r| \ge 2s$ implies $|t-r-s| \ge \frac 12 |t-r|$.  If $h < 1$, then we take $\gamma = 2-h =5/3$, and we have an upper bound of
\[ Cs^{h+2-h}|t-r|^{h-2} = Cs^{h+\gamma}|t-r|^{-\gamma}.\]

\medskip
\noindent{\em Condition (vi)}.  First assume $h=1/3$.  Referring to condition (v) above, we can decompose $\beta_n(j,k)$ as
\[\beta_n(j,k) = \frac{1}{2n^h} \omega(j,k,1) + \frac{1}{2n^h}\psi(j-k,1),\]
where $\omega(j,k,1) = -(j+k+2)^h + 2(j+k+1)^h -(j+k)^h$ and $\psi(j-k,1) = |j-k+1|^h - 2|j-k|^h +|j-k-1|^h$.  Note that $\psi(j-k,1)$ is identical to the $\psi$ used in Prop. 4.1, where in this case $h = 2HK$.  Following the proof of Prop. 4.2, it is enough to show
\beq{omega_0} \lim_{n \to \infty} n^{-1} \sum_{j,k=0}^{\Nt -1} \left|\omega(j,k,1)\right| = 0,\eeq  
so that, similar to \req{eta_is_psi} in the proof of Prop. 4.2, we have
\[\eta(t) = \lim_{n \to \infty} \frac{1}{8n^{3h}}\sum_{j,k=0}^{\Nt -1} \psi(j-k,1)^3 = C_ht,\] where
\[C_h = 1 + \frac 14\sum_{m=1}^\infty\left( (m+1)^{\frac 13} - 2m^{\frac 13} +(m-1)^{\frac 13}\right)^3.\]
That is, $C_h$ corresponds to the constant $C_K$ from Prop. 4.2 with $K=1$.   

\medskip
{\em Proof of \req{omega_0}}.  By Mean Value and the above computation for condition (v), $|\omega(j,k,1)| \le C(j+k)^{-\gamma}$ for some $\gamma > 1$.  Hence, for each $j \ge 2$, 
\begin{align*}
\sum_{k=0}^{\Nt -1} |\omega(j,k,1)| &\le C\sum_{k=0}^{\Nt -1} (j+k)^{-\gamma}\\
&\le C\int_{j-1}^\infty u^{\gamma} du \le C(j-1)^{1-\gamma}.\end{align*}
It follows that we have
\begin{align*}
n^{-1} \sum_{j,k=0}^{\Nt -1} |\omega(j,k,1)| &= n^{-1}\sum_{k=0}^{\Nt-1} \left( |\omega(0,k,1)| + |\omega(1,k,1)| \right)+ n^{-1}\sum_{j=2}^{\Nt -1}\sum_{k=0}^{\Nt -1} |\omega(j,k,1)|\\
&=n^{-1} \sum_{k=0}^{\Nt -1} \left(\left[ (k+2)^h - 2(k+1)^h + k^h\right] + \left[ (k+3)^h - 2(k+2)^h+(k+1)^h\right]\right) \\
&\qquad\;+ Cn^{-1}\sum_{j=2}^{\Nt -1} (j-1)^{1-\gamma}\\
&\le Cn^{-1} + Cn^{-1}\Nt^{2-\gamma}\end{align*}
which converges to 0 since $\gamma > 1$. 
\end{proof}

\section{Proof of Some Technical Lemmas}

\subsection{Proof of Lemma 3.3}
We may assume $t_1 = 0$.  For this proof we use Malliavin calculus to represent $\Delta X_{\frac jn}^5$ as a Skorohod integral.  Consider the Hermite polynomial identity $x^5 = H_5(x) + 10H_3(x) + 15 H_1(x)$.  Using the isometry $H_p(X(h)) = \delta^p(h^{\otimes p})$ (when $\|h\|_\hten =1$) we obtain for each $0 \le j \le \lfloor nt_2 \rfloor -1$,
\beq{dX5} \Delta X_{\frac jn}^5 = \delta^5(\partial_{\frac jn}^{\otimes 5}) + 10\| \Delta X_{\frac jn}\|_{L^2}^2\delta^3(\partial_{\frac jn}^{\otimes 3}) + 15\| \Delta X_{\frac jn}\|_{L^2}^4\delta(\partial_{\frac jn}).
\eeq
With this representation, we can expand
\[\sum_{j,k=0}^{\lfloor nt_2\rfloor -1} {\mathbb E}\left[ f^{(5)}(\hat{X}_{\frac jn})f^{(5)}(\hat{X}_{\frac kn}) \Delta X_{\frac jn}^5 \Delta X_{\frac kn}^5\right]\]
into 9 sums of the form
\beq{DsumX5} C\sum_{j,k=0}^{\lfloor nt_2\rfloor -1} \|\Delta X_{\frac jn}\|_{L^2}^{5-p}\|\Delta X_{\frac kn}\|_{L^2}^{5-q}{\mathbb E}\left[ f^{(5)}(\hat{X}_{\frac jn})f^{(5)}(\hat{X}_{\frac kn})
\delta^p(\partial_{\frac jn}^{\otimes p})
\delta^q(\partial_{\frac kn}^{\otimes q})\right]\eeq
where $p,q$ take values 1, 3, or 5.  By the integral multiplication formula \req{intmult}; 
and using the Malliavin duality \req{Duality}, each term of the form \req{DsumX5} can be further expanded into terms of the form
$$ C\sum_{j,k=0}^{\lfloor nt_2\rfloor -1} \left< \partial_{\frac jn}, \partial_{\frac kn}\right>^r_\hten \|\Delta X_{\frac jn}\|_{L^2}^{5-p}\|\Delta X_{\frac kn}\|_{L^2}^{5-q}{\mathbb E}\left[ f^{(5)}(\hat{X}_{\frac jn})f^{(5)}(\hat{X}_{\frac kn})
\delta^{p+q-2r}(\partial_{\frac jn}^{\otimes p-r}\otimes \partial_{\frac kn}^{\otimes q-r})\right]$$
$$=C\sum_{j,k=0}^{\lfloor nt_2\rfloor -1} \left< \partial_{\frac jn}, \partial_{\frac kn}\right>^r_\hten \|\Delta X_{\frac jn}\|_{L^2}^{5-p}\|\Delta X_{\frac kn}\|_{L^2}^{5-q}{\mathbb E}\left[ \left< D^{p+q-2r}\left(f^{(5)}(\hat{X}_{\frac jn})f^{(5)}(\hat{X}_{\frac kn})\right), \partial_{\frac jn}^{\otimes p-r}\otimes \partial_{\frac kn}^{\otimes q-r}\right>_{\hten^{\otimes p+q-2r}}\right]$$
where $0 \le r \le p\wedge q$ and $p,q \in \{ 1, 3, 5\}$.  For $0 \le m = p+q-2r \le 10$, we have 
\begin{align*}
D^m \left[f^{(5)}(\hat{X}_{\frac jn})f^{(5)}(\hat{X}_{\frac kn})\right] &= \sum_{a+b =m} D^a\left(f^{(5)}(\hat{X}_{\frac jn})\right)D^b\left(f^{(5)}(\hat{X}_{\frac kn})\right)\\
&=\sum_{a+b=m}f^{(5+a)}(\hat{X}_{\frac jn})f^{(5+b)}(\hat{X}_{\frac{j}{n}})\hat{\varepsilon}_{\frac jn}^{\otimes a} \otimes \hat{\varepsilon}_{\frac{k}{n}}^{\otimes b}.
\end{align*}
Hence, we expand \req{DsumX5} again into terms of the form:
\begin{align*}&C\sum_{j,k=0}^{\lfloor nt_2\rfloor -1} \left< \partial_{\frac jn}, \partial_{\frac kn}\right>^r_\hten \|\Delta X_{\frac jn}\|_{L^2}^{5-p}\|\Delta X_{\frac kn}\|_{L^2}^{5-q}{\mathbb E}\left[ f^{(5+a)}(\hat{X}_{\frac jn})f^{(5+b)}(\hat{X}_{\frac kn})\right]\\
&\quad\times\left<\hat{\varepsilon}_{\frac{j}{n}}^{\otimes a} \otimes \hat{\varepsilon}_{\frac{k}{n}}^{\otimes b}, \partial_{\frac jn}^{\otimes p-r} \otimes \partial_{\frac kn}^{\otimes q-r}\right>_{\hten^{\otimes a+b}},\end{align*}
where $a+b = p+q-2r$.  With this representation, we are now ready to develop estimates for each term.  By condition (0),
$$\left|{\mathbb E}\left[ f^{(5+a)}(\hat{X}_{\frac jn})f^{(5+b)}(\hat{X}_{\frac kn})\right]\right| \le \left( \sup_{0\le j < \lfloor nt_2\rfloor} {\mathbb E}\left[f^{(5+a)}(\hat{X}_{\frac jn})\right]\right)^{\frac 12}\left(\sup_{0\le k < \lfloor nt_2\rfloor} {\mathbb E}\left[f^{(5+b)}(\hat{X}_{\frac kn})\right]\right)^{\frac 12} \le C;$$
and by condition (i),
$$ \sup_{0\le j < \lfloor nt_2 \rfloor} \|\Delta X_{\frac jn}\|_{L^2}^{5-p}~\sup_{0\le k < \lfloor nt_2 \rfloor} \|\Delta X_{\frac kn}\|_{L^2}^{5-q} \le Cn^{-\frac{10-(p+q)}{6}}. $$
If $a \ge 1$ with $a+b=p+q-2r$, by condition (iv)
$$\left|\left<\hat{\varepsilon}_{\frac jn}^{\otimes a} \otimes \hat{\varepsilon}_{\frac kn}^{\otimes b}, \partial_{\frac jn}^{\otimes p-r} \otimes \partial_{\frac kn}^{\otimes q-r}\right>_{\hten^{\otimes a+b}}\right|\le Cn^{-(a+b-1)\lambda}\left|\left< \hat{\varepsilon}_{\frac jn}, \partial_{\frac jn}\right>_\hten\right|,$$
with a similar term in $k$ if $a =0$ and $b\ge 1$.  Hence, assuming $a \ge 1$, each term in the expansion of \req{DsumX5} has an upper bound of
$$C\sum_{j,k=0}^{\lfloor nt_2\rfloor -1} \left|\left< \partial_{\frac jn}, \partial_{\frac kn}\right>_\hten\right|^r  n^{-\frac{10-(p+q)}{6}}~ n^{-(a+b-1)\lambda}\left|\left< \hat{\varepsilon}_{\frac jn}, \partial_{\frac jn}\right>_\hten\right|.$$
To show each term has the desired upper bound, first assume $r \ge 1$.  Then $\left|\left< \hat{\varepsilon}_{\frac jn}, \partial_{\frac jn}\right>_\hten\right| \le Cn^{-\lambda},$ and by Lemma 3.2.d we have an upper bound of 
$$Cn^{-\frac{10-(p+q)}{6} - (a+b)\lambda}\sum_{j,k=0}^{\lfloor nt_2 \rfloor -1} \left| \left< \partial_{\frac jn}, \partial_{\frac kn}\right>_\hten\right|^r \le C\lfloor nt_2 \rfloor n^{-\frac 53 +\frac{p+q-2r}{6} -(a+b)\lambda} = C\lfloor nt_2 \rfloor n^{-\frac 53 -(a+b)(\lambda -\frac 16)} $$
which is less than or equal to $C\lfloor nt_2 \rfloor n^{-\frac 53}$ because $\lambda > 1/6$.  For cases with $r=0$, then either $a \ge 1$ or $b\ge 1$, so without loss of generality assume $a \ge 1.$  For this case with Lemma 3.2.c we have an upper bound of
$$C\sum_{j,k=0}^{\lfloor nt_2 \rfloor-1} n^{-\frac{10-(p+q)}{6} - (a+b-1)\lambda}\left|\left<\hat{\varepsilon}_{\frac jn} , \partial_{\frac jn}\right>_\hten\right|  \le C\lfloor nt_2 \rfloor n^{-\frac 53 -(a+b)(\lambda -\frac 16) + \lambda},$$
which is less than $C\lfloor nt_2 \rfloor n^{-\frac 43}$ since $\lambda \le 1/3$.

\subsection{Proof of Lemma 3.6}
Without loss of generality, assume $a=0$.  First we want to show that for each integer $0 \le k \le b-1$,
\beq{single} {\mathbb E}\left| \sum_{j=0}^k f^{(3)}(\hat{X}_{\frac jn}) \Delta X_{\frac jn} \right| \le C.\eeq
Using the Taylor expansion similar to Section 3.2, 
\begin{align*}
f''(X_{\frac{j+1}{n}}) - f''(X_{\frac jn}) &= \left(f''(X_{\frac{j+1}{n}}) - f''(\hat{X}_{\frac jn})\right) -\left(f''(X_{\frac{j}{n}}) - f''(\hat{X}_{\frac jn})\right) \\
&= f^{(3)}(\hat{X}_{\frac jn})\Delta X_{\frac jn} + \frac{1}{24}f^{(5)}(\Xhat) \Delta X_{\frac jn}^3 + \frac{1}{2^5 5!}f^{(7)}(\Xhat)\Delta X_{\frac jn}^5 + B_n^+(j) - B_n^-(j)
\end{align*}
where $B_n^+(j), B_n^-(j)$ have the form $Cf^{(9)}(\xi_j) \Delta X_{\frac jn}^7$.  Hence we can write,
\begin{align*}
{\mathbb E}\left| \sum_{j=0}^k f^{(3)}(\hat{X}_{\frac jn}) \Delta X_{\frac jn} \right| & \le {\mathbb E}\left| \sum_{j=0}^k \left(f''(X_{\frac{j+1}{n}}) - f''(X_{\frac jn})\right)\right| + \frac{1}{24} {\mathbb E}\left| \sum_{j=0}^k f^{(5)}(\Xhat)\Delta X_{\frac jn}^3\right| \\
&\quad + \frac{1}{2^5 5!}{\mathbb E}\left|\sum_{j=0}^k f^{(7)}(\Xhat)\Delta X_{\frac jn}^5\right| + {\mathbb E}\sum_{j=0}^k \left|B_n^+(j)\right| + \left|B_n^-(j)\right|.
\end{align*}
We have the following estimates:  By condition (0),
$${\mathbb E}\left| \sum_{j=0}^k \left(f''(X_{\frac{j+1}{n}}) - f''(X_{\frac jn})\right)\right| \le {\mathbb E}\left| f''(X_{\frac{k+1}{n}}) - f''(X_0)\right| \le C;$$
by Lemma 3.3,
$${\mathbb E}\left| \sum_{j=0}^k f^{(7)}(\Xhat)\Delta X_{\frac jn}^5\right| \le Cn^{-\frac 23}(k+1)^{\frac 12};$$
and by Lemma 3.4, 
$$\sum_{j=0}^k {\mathbb E}\left|B_n^+(j)\right| + {\mathbb E}\left|B_n^-(j)\right| \le Cn^{-\frac 73}(k+1)^2.$$
This leaves the $\Delta X^3$ term.  Using the Hermite polynomial identity $y^3 = H_3(y) + 3H_1(y)$, we can write  
\begin{align*}
{\mathbb E}\left| \sum_{j=0}^k f^{(5)}(\Xhat)\Delta X_{\frac jn}^3\right| \le {\mathbb E}\left| \sum_{j=0}^k f^{(5)}(\Xhat)\delta^3(\partial_{\frac jn}^{\otimes 3})\right| + {\mathbb E}\left| \sum_{j=0}^k \left\| \Delta X_{\frac jn}\right\|_{L^2}^2 ~f^{(5)}(\Xhat)\delta(\partial_{\frac jn})\right|.
\end{align*}
For the first term we have 
\begin{align*}&{\mathbb E}\left| \sum_{j=0}^k f^{(5)}(\Xhat)\delta^3(\partial_{\frac jn}^{\otimes 3})\right|^2\\
&\qquad = \sum_{j,\ell = 0}^k{\mathbb E}\left[ f^{(5)}(\Xhat)f^{(5)}(\hat{X}_{\frac{\ell}{n}}) \delta^3\left(\partial_{\frac jn}^{\otimes 3}\right)\delta^3\left(\partial_{\frac{\ell}{n}}^{\otimes 3}\right)\right]\\
&\qquad=\sum_{j,\ell = 0}^k \sum_{r=0}^3 r!{\binom r3}^2{\mathbb E}\left[ f^{(5)}(\Xhat)f^{(5)}(\hat{X}_{\frac{\ell}{n}}) \delta^{6-2r}\left(\partial_{\frac jn}^{\otimes 3-r}\otimes_r \partial_{\frac{\ell}{n}}^{\otimes 3-r}\right)\right]\\
&\qquad=\sum_{j,\ell = 0}^k \sum_{r=0}^3 r!{\binom r3}^2{\mathbb E}\left[ \left< D^{6-2r}\left[f^{(5)}(\Xhat)f^{(5)}(\hat{X}_{\frac{\ell}{n}})\right] ,\partial_{\frac jn}^{\otimes 3-r}\otimes_r \partial_{\frac{\ell}{n}}^{\otimes 3-r}\right>_{\hten^{\otimes 6-2r}}\right]\left< \partial_{\frac jn}, \partial_{\frac{\ell}{n}}\right>_\hten^r\\
&\qquad\le \sum_{j,\ell = 0}^k \sum_{r=0}^3 \sum_{\stackrel{a+b=}{6-2r}} \left|\left< \hat{\varepsilon}_{\frac jn}^{\otimes a} \otimes \hat{\varepsilon}_{\frac{\ell}{n}}^{\otimes b}, \partial_{\frac jn}^{\otimes 3-r} \otimes \partial_{\frac{\ell}{n}}^{\otimes 3-r}\right>_{\hten^{\otimes 6-2r}}\right|~\left|\left< \partial_{\frac jn}, \partial_{\frac{\ell}{n}}\right>_\hten\right|^r.
\end{align*}
For this sum, if $r=0$ we use Lemma 3.2.a and 3.2.b for each pair $(a,b)$ to obtain terms of the form
\begin{align*}\sum_{j,\ell = 0}^k  \left|\left< \hat{\varepsilon}_{\frac jn}^{\otimes a} \otimes \hat{\varepsilon}_{\frac{\ell}{n}}^{\otimes b}, \partial_{\frac jn}^{\otimes 3} \otimes \partial_{\frac{\ell}{n}}^{\otimes 3}\right>_{\hten^{\otimes 6}}\right|
&\le \sup_{j,\ell}\left|\left< \hat{\varepsilon}_{\frac jn}, \partial_{\frac{\ell}{n}}\right>_\hten\right|^3 \sup_{j}\left|\left< \hat{\varepsilon}_{\frac jn}, \partial_{\frac jn}\right>_\hten\right| \sum_{j,\ell =0}^k \left< \hat{\varepsilon}_{\frac jn}, \partial_{\frac{\ell}{n}}\right>_\hten^2 \\ 
 &\le Cn^{-1-3\lambda}(k+1),\end{align*}
 where we use the fact that $r=0$ implies $a \ge 3$ or $b\ge 3$.  If $r \ge 1$, we use Lemma 3.2.a and 3.2.d to obtain terms of the form
\[Cn^{-(6-2r)\lambda}\sum_{j,\ell=0}^k \left|\left< \partial_{\frac jn}, \partial_{\frac{\ell}{n}}\right>_\hten\right|^r \le Cn^{-(6-2r)\lambda -\frac r3}(k+1),\]
noting that $(6-2r)\lambda + \frac r3 > 1$.  

For the other term, we have by Lemma 3.2.b,
\begin{align*}
&{\mathbb E}\left( \sum_{j=0}^k \left\| \Delta X_{\frac jn}\right\|_{L^2}^2 ~f^{(5)}(\Xhat)\delta(\partial_{\frac jn})\right)^2\\ &\quad\le
\sup_{0\le j \le k}\left\| \Delta X_{\frac jn}\right\|_{L^2}^4 \sum_{j,\ell =0}^k \left|{\mathbb E}\left[ f^{(5)}(\Xhat)f^{(5)}(\hat{X}_{\frac{\ell}{n}})\left(\delta^2\left(\partial_{\frac jn}\otimes\partial_{\frac{\ell}{n}}\right) + \left< \partial_{\frac jn}, \partial_{\frac{\ell}{n}}\right>_\hten\right)\right]\right|\\
&\quad\le Cn^{-\frac 23} \sum_{j,\ell =0}^k \left|{\mathbb E}\left[\left< D^2\left[f^{(5)}(\Xhat)f^{(5)}(\hat{X}_{\frac{\ell}{n}})\right], \partial_{\frac jn}\otimes\partial_{\frac{\ell}{n}}\right>_{\hten^{\otimes 2}}\right]\right|+\left|{\mathbb E}\left[ f^{(5)}(\Xhat)f^{(5)}(\hat{X}_{\frac{\ell}{n}})\right]\left< \partial_{\frac jn}, \partial_{\frac{\ell}{n}}\right>_\hten\right|\\
&\quad\le Cn^{-\frac 23} \sum_{j,\ell =0}^k \,\sum_{a +b =2} \left|\left< \hat{\varepsilon}_{\frac jn}^{\otimes a} \otimes \hat{\varepsilon}_{\frac{\ell}{n}}^{\otimes b},\partial_{\frac jn}\otimes\partial_{\frac{\ell}{n}}\right>_{\hten^{\otimes 2}}\right|+Cn^{-\frac 23}\sum_{j,\ell=0}^k \left|\left< \partial_{\frac jn}, \partial_{\frac{\ell}{n}}\right>_\hten\right|\\
&\quad\le Cn^{-\frac 23} \left(\sum_{j =0}^k \left|\left< \hat{\varepsilon}_{\frac jn}, \partial_{\frac jn}\right>_\hten\right|\right)\left(\sum_{\ell=0}^k \left|\left<\hat{\varepsilon}_{\frac jn},\partial_{\frac{\ell}{n}}\right>_\hten\right|\right) +\left(\sum_{j =0}^k \left|\left< \hat{\varepsilon}_{\frac jn}, \partial_{\frac jn}\right>_\hten\right|\right)^2+ Cn^{-1}(k+1)\\
&\quad \le Cn^{-1}(k+1)^{1-\theta} + Cn^{-\frac 43}(k+1)^{2-2\theta} + Cn^{-1}(k+1) \le C,\end{align*}
where the estimates follow from Lemma 3.2.c and 3.2.d.  Hence, by Cauchy-Schwarz
$${\mathbb E}\left| \sum_{j=0}^k \left\| \Delta X_{\frac jn}\right\|_{L^2}^2 ~f^{(5)}(\Xhat)\delta(\partial_{\frac jn})\right|
\le C,$$
which proves \req{single}.  Now we define
$$G_n(j) = \sum_{k=0}^j f^{(3)}(\hat{X}_{\frac kn}) \Delta X_{\frac kn},$$
and by Abel's formula and condition (iii) we have
\begin{align*}
{\mathbb E}\left|\sum_{j=0}^{b-1} \left\| \Delta X_{\frac jn}\right\|_{L^2}^2 f^{(3)}(\Xhat)\Delta X_{\frac jn}\right| 
&\le \left\| \Delta X_{\frac bn}\right\|_{L^2}^2{\mathbb E}\left| G_n(b-1)\right| +  \sum_{j=0}^{b-1} {\mathbb E}\left|G_n(j)\right|\left(\left\| \Delta X_{\frac{j+1}{n}}\right\|_{L^2}^2 - \left\| \Delta X_{\frac jn}\right\|_{L^2}^2\right)\\
&\le Cn^{-\frac 13}+ Cn^{-\frac 13}\sum_{j=4}^{b-1}(j-1)^{-\nu}\\
&\le Cn^{-\frac 13}.\end{align*}

\subsection{Proof of Lemma 3.10}
\noindent{\em Proof of \req{D2F}}.  Let $a_j = \lfloor nt_{j-1} \rfloor$ and $b_j = \lfloor nt_j \rfloor$.  By Lemma 2.1.b,
\begin{align*} D^2 F_n^j &= \sum_{k= a_j}^{b_j -1} D^2\delta^3 \left( f^{(3)}(\hat{X}_{\frac kn})\partial_{\frac kn}^{\otimes 3}\right) \\
&= \sum_{k= a_j}^{b_j -1} \left\{\delta^3 \left( f^{(5)}(\hat{X}_{\frac kn})\partial_{\frac kn}^{\otimes 3}\right)\hat{\varepsilon}_{\frac kn}^{\otimes 2} + 6~\delta^2 \left( f^{(4)}(\hat{X}_{\frac kn})\partial_{\frac kn}^{\otimes 2}\right)\partial_{\frac kn}\otimes\hat{\varepsilon}_{\frac kn} + 6~\delta \left( f^{(3)}(\hat{X}_{\frac kn})\partial_{\frac kn}\right)\partial_{\frac kn}^{\otimes 2} \right\}
\end{align*}
and 
\beq{DFexp}
DF_n^k = \sum_{m = a_k}^{b_k -1} \left\{ \delta^3\left(f^{(4)}(\hat{X}_{\frac mn})\partial_{\frac mn}^{\otimes 3}\right)\hat{\varepsilon}_{\frac mn} + 3~\delta^2\left( f^{(3)}(\hat{X}_{\frac mn})\partial_{\frac mn}^{\otimes 2}\right)\partial_{\frac mn}\right\}.\eeq
With these two expansions, it follows that the expectation
$$ {\mathbb E}\left[ \left< u_n^i, D^2F_n^j \otimes DF_n^k \right>_{\hten^{\otimes 3}}^2\right]$$
consists of terms of the form
\[
\sum_{p,p' = a_i}^{b_i-1} ~\sum_{q,q' = a_j}^{b_j-1}~\sum_{m,m' = a_k}^{b_k-1} {\mathbb E}\left[ G(p,p')\delta^{r_1}\left( g_1(\hat{X}_{\frac qn})\partial_{\frac qn}^{\otimes r_1}\right)\delta^{r_2}\left( g_2(\hat{X}_{\frac{q'}{n}})\partial_{\frac{q'}{n}}^{\otimes r_2}\right)\delta^{r_3}\left( g_3(\hat{X}_{\frac mn})\partial_{\frac mn}^{\otimes r_3}\right)\delta^{r_4}\left( g_4(\hat{X}_{\frac{m'}{n}})\partial_{\frac{m'}{n}}^{\otimes r_4}\right)\right]\]
\[ \times \left< \hat{\varepsilon}_{\frac qn}, \partial_{\frac pn}\right>_\hten^{r_1-1} \left<\partial_{\frac qn}, \partial_{\frac pn}\right>_\hten^{3-r_1} \left<\hat{\varepsilon}_{\frac{q'}{n}}, \partial_{\frac{p'}{n}}\right>_\hten^{r_2-1}\left<\partial_{\frac{q'}{n}}, \partial_{\frac{p'}{n}}\right>_\hten^{3-r_2}\left< \hat{\varepsilon}_{\frac mn}, \partial_{\frac pn}\right>_\hten^{r_3-2}\] \beq{six_sum}\times \left<\partial_{\frac mn}, \partial_{\frac pn}\right>_\hten^{3-r_3} \left<\hat{\varepsilon}_{\frac{m'}{n}}, \partial_{\frac{p'}{n}}\right>_\hten^{r_4-2}\left<\partial_{\frac{m'}{n}}, \partial_{\frac{p'}{n}}\right>_\hten^{3-r_4}
\eeq
where $G(p,p') := f^{(3)}(\hat{X}_{\frac pn}) f^{(3)}(\hat{X}_{\frac{p'}{n}})$, $r_1, r_2$ take values 1,2 or 3; $r_3, r_4$ take values 2 or 3; and each $g_i$ represents the appropriate derivative of $f$.  Without loss of generality, we will assume that $u_n^i, D^2F_n^j$, and $DF_n^k$ are all defined over the interval $[0,t]$, and that all sums are over the set $\{ 0, \dots, \Nt-1\}$.  Let $R = r_1 + r_2 + r_3 + r_4$, and note that $6 \le R \le 12$.  It follows from Lemma 3.2.a, 3.2.c, and/or 3.2.d that
\begin{align*}&\sum_{p,p'=0}^{\Nt -1} \vert\left< \hat{\varepsilon}_{\frac qn}, \partial_{\frac pn}\right>_\hten^{r_1-1} \left<\partial_{\frac qn}, \partial_{\frac pn}\right>_\hten^{3-r_1} \left<\hat{\varepsilon}_{\frac{q'}{n}}, \partial_{\frac{p'}{n}}\right>_\hten^{r_2-1}\left<\partial_{\frac{q'}{n}}, \partial_{\frac{p'}{n}}\right>_\hten^{3-r_2}\left< \hat{\varepsilon}_{\frac mn}, \partial_{\frac pn}\right>_\hten^{r_3-2}\\
&\qquad \times \left<\partial_{\frac mn}, \partial_{\frac pn}\right>_\hten^{3-r_3} \left<\hat{\varepsilon}_{\frac{m'}{n}}, \partial_{\frac{p'}{n}}\right>_\hten^{r_4-2}\left<\partial_{\frac{m'}{n}}, \partial_{\frac{p'}{n}}\right>_\hten^{3-r_4} \vert\\
&\quad\le\sum_{p=0}^{\Nt -1}\left|\left< \hat{\varepsilon}_{\frac qn}, \partial_{\frac pn}\right>_\hten^{r_1-1} \left<\partial_{\frac qn}, \partial_{\frac pn}\right>_\hten^{3-r_1} \left< \hat{\varepsilon}_{\frac mn}, \partial_{\frac pn}\right>_\hten^{r_3-2}\left<\partial_{\frac mn}, \partial_{\frac pn}\right>_\hten^{3-r_3} \right|\\
&\qquad \times \sum_{p'=0}^{\Nt -1} \left| \left<\hat{\varepsilon}_{\frac{q'}{n}}, \partial_{\frac{p'}{n}}\right>_\hten^{r_2-1}\left<\partial_{\frac{q'}{n}}, \partial_{\frac{p'}{n}}\right>_\hten^{3-r_2}\left<\hat{\varepsilon}_{\frac{m'}{n}}, \partial_{\frac{p'}{n}}\right>_\hten^{r_4-2}\left<\partial_{\frac{m'}{n}}, \partial_{\frac{p'}{n}}\right>_\hten^{3-r_4}\right|\\
&\quad \le Cn^{-\Lambda},
\end{align*}
where the exponent $\Lambda$ is determined by $\{r_1, \dots, r_4\}$ as follows:  First, suppose $r_1=3$.  Then by Lemma 3.2.a and 3.2.c,
\begin{align*}
\sum_{p=0}^{\Nt -1}\left|\left< \hat{\varepsilon}_{\frac qn}, \partial_{\frac pn}\right>_\hten^{2}\left< \hat{\varepsilon}_{\frac mn}, \partial_{\frac pn}\right>_\hten^{r_3-2}\left<\partial_{\frac mn}, \partial_{\frac pn}\right>_\hten^{3-r_3} \right|
&\le \sup_{m,p}\left|\left< \hat{\varepsilon}_{\frac mn}, \partial_{\frac pn}\right>_\hten^{r_3-2}\left<\partial_{\frac mn}, \partial_{\frac pn}\right>_\hten^{3-r_3} \right|\sum_{p=0}^{\Nt -1}\left|\left< \hat{\varepsilon}_{\frac qn}, \partial_{\frac pn}\right>_\hten^{2}\right|\\
&\le Cn^{-2\lambda - (r_3-2)\lambda -\frac 13(3-r_3)}\\
&= Cn^{-(r_1 + r_3 -3)\lambda - \frac 13(6-r_1-r_3)}.\end{align*}
On the other hand, if $r_1 =1$ or 2 then by Lemma 3.2.a and 3.2.d,
\begin{align*}
&\sum_{p=0}^{\Nt -1}\left|\left< \hat{\varepsilon}_{\frac qn}, \partial_{\frac pn}\right>_\hten^{r_1-1} \left<\partial_{\frac qn}, \partial_{\frac pn}\right>_\hten^{3-r_1} \left< \hat{\varepsilon}_{\frac mn}, \partial_{\frac pn}\right>_\hten^{r_3-2}\left<\partial_{\frac mn}, \partial_{\frac pn}\right>_\hten^{3-r_3} \right|\\
&\qquad \le \sup_{q,p}\left|\left< \hat{\varepsilon}_{\frac qn}, \partial_{\frac pn}\right>_\hten^{r_1-1}\right| \sup_{m,p}\left|\left< \hat{\varepsilon}_{\frac mn}, \partial_{\frac pn}\right>_\hten^{r_3-2}\left<\partial_{\frac mn}, \partial_{\frac pn}\right>_\hten^{3-r_3} \right|\sum_{p=0}^{\Nt -1} \left| \left<\partial_{\frac qn}, \partial_{\frac pn}\right>_\hten^{3-r_1}\right|\\
&\qquad \le Cn^{-(r_1+r_3-3)\lambda -\frac 13(6-r_1-r_3)}.\end{align*}
Combining this with a similar computation for the sum over $p'$, we obtain
\[ \Lambda = \lambda(R-6) +\frac 13 (12-R) = 2-\left(\frac 13 - \lambda\right)(R-6).\] 
In particular, $\Lambda = 2$ if $R=6$ and $\Lambda =\frac 53+\lambda$ for $R=7$.  It follows that we want to find bounds for terms of the form
\beq{uD2D} Cn^{-\Lambda}\sup_{p,p'} \sum_{j_1, j_2, j_3, j_4} \left|{\mathbb E}\left[G(p,p')\prod_{i=1}^4 \delta^{r_i}\left(g_i(\hat{X}_{\frac{j_i}{n}})\partial_{\frac{j_i}{n}}^{\otimes r_i}\right)\right]\right|.\eeq
By repeated use of \req{intmult}, we can expand each product of the form
\[\prod_{i=1}^4 \delta^{r_i}\left(g_i(\hat{X}_{\frac{j_i}{n}})\partial_{\frac{j_i}{n}}^{\otimes r_i}\right)\]
into a sum of terms of the form
\[ C_M\delta^M\left(\Psi_n \partial_{\frac{j_1}{n}}^{\otimes b_1}\otimes \partial_{\frac{j_2}{n}}^{\otimes b_2}\otimes\partial_{\frac{j_3}{n}}^{\otimes b_3}\otimes\partial_{\frac{j_4}{n}}^{\otimes b_4}\right)\left< \partial_{\frac{j_1}{n}}, \partial_{\frac{j_2}{n}}\right>_\hten^{\alpha_1}\left< \partial_{\frac{j_1}{n}}, \partial_{\frac{j_3}{n}}\right>_\hten^{\alpha_2}\left< \partial_{\frac{j_1}{n}}, \partial_{\frac{j_4}{n}}\right>_\hten^{\alpha_3}\]
\[\times \left< \partial_{\frac{j_2}{n}}, \partial_{\frac{j_3}{n}}\right>_\hten^{\alpha_4}\left< \partial_{\frac{j_2}{n}}, \partial_{\frac{j_4}{n}}\right>_\hten^{\alpha_5}\left< \partial_{\frac{j_3}{n}}, \partial_{\frac{j_4}{n}}\right>_\hten^{\alpha_6},\]
where $C_M$ is a combinatorial constant from \req{intmult}, $\Psi_n = \prod_{i=1}^4 g_i(\hat{X}_{\frac{j_i}{n}})$; each $\alpha_i \in \{0,1,2\}$, such that $A := \sum_{i=1}^6 \alpha_i \le R/2$; each nonnegative integer $b_i$ satisfies $b_i \le r_i$; and the exponent $M$ satisfies:
\[M = b_1+b_2+b_3+b_4  = R-2A.\]
With this representation, and using the Malliavin duality \req{Duality}, we want to bound terms of the form
\beq{uD2D_2} Cn^{-\Lambda} \sup_{p,p'} \sum_{j_1, j_2,j_3,j_4} \left|{\mathbb E}\left[\left< D^MG(p,p'), \Psi_n \partial_{\frac{j_1}{n}}^{\otimes b_1}\otimes \cdots\otimes\partial_{\frac{j_4}{n}}^{\otimes b_4}\right>_{\hten^{\otimes M}}\right]\left< \partial_{\frac{j_1}{n}}, \partial_{\frac{j_2}{n}}\right>_\hten^{\alpha_1}\cdots\left< \partial_{\frac{j_3}{n}}, \partial_{\frac{j_4}{n}}\right>_\hten^{\alpha_6}\right|.\eeq

Consider first the case $A=0$.  Then $M=R \ge 6$, and each $b_i = r_i \ge 1$.  Hence
\begin{align*}
&Cn^{-\Lambda} \sup_{p,p'} \sum_{j_1, j_2,j_3,j_4} \left|{\mathbb E}\left[\left< D^MG(p,p'), \Psi_n \partial_{\frac{j_1}{n}}^{\otimes b_1}\otimes \cdots\otimes\partial_{\frac{j_4}{n}}^{\otimes b_4}\right>_{\hten^{\otimes M}}\right]\right|\\
&\qquad \le Cn^{-\Lambda} \sup_{p,j}\left|\left<\hat{\varepsilon}_{\frac pn}, \partial_{\frac jn}\right>_\hten\right|^{R-4}\left(\sup_p \sum_{j=0}^{\Nt -1} \left|\left<\hat{\varepsilon}_{\frac pn}, \partial_{\frac jn}\right>_\hten\right|\right)^4
\end{align*}
By Lemma 3.2.a and 3.2.c, this is bounded by $Cn^{-1-2\lambda}$, since $\Lambda \ge 1$ for all $R$ and $R\ge 6$.   

If $A\ge 1$, by permutation of indices we may assume that $\alpha_1 \ge 1$, so \req{uD2D_2} may be bounded using Lemma 3.2.c and 3.2.d:
\begin{align*}
&Cn^{-\Lambda} \sup_{p,p'} \sup_{j_1,j_2} \sum_{j_3, j_4=0}^{\Nt -1} \left|{\mathbb E}\left[\left< D^MG(p,p'), \Psi_n \partial_{\frac{j_1}{n}}^{\otimes b_1}\otimes \cdots\otimes\partial_{\frac{j_4}{n}}^{\otimes b_4}\right>_{\hten^{\otimes M}}\right]\right|\\
&\qquad\quad\times \sup_{j,k}\left| \left<\partial_{\frac jn},\partial_{\frac kn}\right>_\hten \right|^{A-1}\sum_{j_1,j_2=0}^{\Nt -1} \left| \left< \partial_{\frac{j_1}{n}}, \partial_{\frac{j_2}{n}}\right>_\hten\right|\\
&\quad \le C\Nt^3 n^{-\Lambda -\frac 13}\sup_{p,j}\left|\left<\hat{\varepsilon}_{\frac pn},\partial_{\frac jn}\right>_\hten\right|^M \sup_{j,k}\left| \left<\partial_{\frac jn},\partial_{\frac kn}\right>_\hten \right|^{A-1}\\
&\quad \le \Nt^3 n^{-\Theta},\end{align*}
where
\begin{align*}
\Theta &= 4+(R-6+M)\lambda - \frac{R-A}{3} = 4+(R-A)(2\lambda -\frac 13) - 6\lambda.\end{align*}
Since $A \le R/2$, $R\ge 6$, and $\lambda > 1/6$, we have $\Theta > 3$ for all cases except when $R=6$, $A=3$. This case has the form,
\begin{align*}
&Cn^{-2}\sup_{p,p'} \sum_{j_1,j_2,j_3,j_4} \left|{\mathbb E}\left[ G(p,p')\Psi_n\right]\left< \partial_{\frac{j_1}{n}},\partial_{\frac{j_2}{n}}\right>_\hten^{\alpha_1}\cdots\left< \partial_{\frac{j_3}{n}},\partial_{\frac{j_4}{n}}\right>_\hten^{\alpha_6}\right|\\
&\qquad \le Cn^{-2} \sup_{j_1,j_2} \sum_{j_3,j_4}\left|\left< \partial_{\frac{j_1}{n}},\partial_{\frac{j_2}{n}}\right>_\hten^{\alpha_1-1}\left<\partial_{\frac{j_1}{n}},\partial_{\frac{j_3}{n}}\right>_\hten^{\alpha_2}\cdots\left< \partial_{\frac{j_3}{n}},\partial_{\frac{j_4}{n}}\right>_\hten^{\alpha_6}\right| 
\sum_{j_1,j_2=0}^{\Nt -1} \left| \left< \partial_{\frac{j_1}{n}}, \partial_{\frac{j_2}{n}}\right>_\hten\right|\\
&\qquad \le C\Nt n^{-\frac 73}\left[ \left(\sup_k \sum_{j=0}^{\Nt -1} \left|\left< \partial_{\frac jn},\partial_{\frac kn}\right>_\hten\right|\right)^2 + \sup_{j,k}\left|\left< \partial_{\frac jn},\partial_{\frac kn}\right>_\hten\right|\sum_{j_3, j_4=0}^{\Nt -1} \left| \left<\partial_{\frac{j_3}{n}},\partial_{\frac{j_4}{n}}\right>_\hten\right|\right]\\
&\qquad \le C\Nt^2 n^{-3},
\end{align*}
by Lemma 3.2.d.

\medskip
\noindent{\em Proof of \req{DFDFDF}}.  For this term, we see that
$$ {\mathbb E}\left[ \left< u_n^i, DF_n^j \otimes DF_n^k \otimes DF_n^\ell\right>_{\hten^{\otimes 3}}^2\right]$$
consists of terms with the form
\[
\sum_{p,p' = a_i}^{b_i-1} ~\sum_{j_1,j_2 = a_j}^{b_j-1}~\sum_{j_3,j_4 = a_k}^{b_k-1}~\sum_{j_5,j_6=a_\ell}^{b_\ell -1} {\mathbb E}\left[ G(p,p')\delta^{r_1}\left( g_1(\hat{X}_{\frac{j_1}{n}})\partial_{\frac{j_1}{n}}^{\otimes r_1}\right)\cdots\delta^{r_6}\left( g_6(\hat{X}_{\frac{j_6}{n}})\partial_{\frac{m'}{n}}^{\otimes r_6}\right)\right]\]
\[ \times \left< \hat{\varepsilon}_{\frac{j_1}{n}}, \partial_{\frac pn}\right>_\hten^{r_1-2} \left<\partial_{\frac{j_1}{n}}, \partial_{\frac pn}\right>_\hten^{3-r_1} \cdots\left< \hat{\varepsilon}_{\frac{j_6}{n}}, \partial_{\frac pn}\right>_\hten^{r_6-2}\left< \partial_{\frac{j_6}{n}}, \partial_{\frac pn}\right>_\hten^{3-r_6}.\]
where each $r_i \in\{ 2,3\}$ and $G(p,p'), g_i(x)$ are as defined above.  As with \req{D2F} above, we assume that all components are defined over the time interval $[0,t]$ for some $t \le T$.  As above, let $R = \sum_{i=1}^6 r_i$, and note that for this case $12 \le R \le 18$.  Similar to the above case, we obtain
\[ \sum_{p,p'} \left| \left< \hat{\varepsilon}_{\frac{j_1}{n}}, \partial_{\frac pn}\right>_\hten^{r_1-2} \left<\partial_{\frac{j_1}{n}}, \partial_{\frac pn}\right>_\hten^{3-r_1} \cdots\left< \hat{\varepsilon}_{\frac{j_6}{n}}, \partial_{\frac pn}\right>_\hten^{r_6-2}\left< \partial_{\frac{j_6}{n}}, \partial_{\frac pn}\right>_\hten^{3-r_6}\right| \le Cn^{-\Lambda'},\]
where $\Lambda' = 2 - (\frac 13 -\lambda)(R-12)$.  It follows that, similar to \req{uD2D}, we want to obtain bounds for terms of the form
\[ Cn^{-\Lambda'} \sup_{p,p'} \sum_{j_1, \dots, j_6=0}^{\Nt-1} \left| {\mathbb E}\left[ G(p,p')\prod_{i=1}^6 \delta^{r_i}\left(g_i(\hat{X}_{\frac{j_i}{n}})\partial_{\frac{j_i}{n}}^{\otimes r_i}\right)\right]\right|.\]
Using \req{intmult} and the Malliavin duality as before, we obtain terms of the form
\beq{uDDD} Cn^{-\Lambda} \sup_{p,p'} \sum_{j_1, \dots, j_6=0}^{\Nt-1} \left| {\mathbb E}\left[\left< D^MG(p,p), \delta^M\left(\tilde{\Psi}_n \partial_{\frac{j_1}{n}}^{\otimes b_1}\otimes \cdots\otimes\partial_{\frac{j_6}{n}}^{\otimes b_6}\right)\right>_{\hten^{\otimes M}}\right]\right| \prod_{\{j_1, \dots, j_6\}}\left| \left< \partial_{\frac{j_\ell}{n}}, \partial_{\frac{j_m}{n}}\right>_\hten^{\alpha_i}\right|, \eeq
where
$\tilde{\Psi}_n = \prod_{i=1}^6 g_i(\hat{X}_{\frac{j_i}{n}})$, each $\alpha_i$ and each $b_i$ take values from $\{ 0, 1, 2, 3\}$; and the product includes all 15 possible pairs from the set $\{j_1, \dots, j_6\}$ such that $A:=\sum_{i=1}^{15} \alpha_i \le R/2$.  As in the above case, for each $R$ we have $M$ and $A$ satisfying
$M = \sum_{i=1}^6 b_i \;\text{ and }\; M = R-2A$.

In the product
\beq{DDDprod}\left< D^MG(p,p), \delta^M\left(\tilde{\Psi}_n \partial_{\frac{j_1}{n}}^{\otimes b_1}\otimes \cdots\otimes\partial_{\frac{j_6}{n}}^{\otimes b_6}\right)\right>_{\hten^{\otimes M}} \prod_{\{i_1, \dots, i_6\}} \left< \partial_{\frac{j_\ell}{n}}, \partial_{\frac{j_m}{n}}\right>_\hten^{\alpha_i}\eeq
each of the indices $\{ j_1, \dots, j_6\}$ must appear at least once.  Note that by Lemma 3.2.a we have (possibly up to a fixed constant)
\[\sup_{0\le j,k\le \Nt}\left| \left<\partial_{\frac jn},\partial_{\frac kn}\right>_\hten\right| \le \sup_{0\le j,p\le \Nt}\left|\left< \hat{\varepsilon}_{\frac pn}, \partial_{\frac jn}\right>_\hten\right|,\]
and by Lemma 3.2.c and 3.2.d we have
\[\sup_{0\le k \le \Nt} \sum_{j=0}^{\Nt -1} \left|\left<\partial_{\frac jn},\partial_{\frac kn}\right>_\hten\right | \le \sup_{0\le p \le \Nt} \sum_{j=0}^{\Nt -1} \left|\left< \hat{\varepsilon}_{\frac pn}, \partial_{\frac jn}\right>_\hten\right|.\]
Hence, we may conclude that \req{DDDprod} contains terms less than or equal to
\[ \left<\hat{\varepsilon}_{\frac pn}, \partial_{\frac{j_1}{n}}\right>_\hten\left<\hat{\varepsilon}_{\frac pn}, \partial_{\frac{j_3}{n}}\right>_\hten\left<\hat{\varepsilon}_{\frac pn}, \partial_{\frac{j_3}{n}}\right>_\hten,\]
and, by Lemma 3.2.c, \req{uDDD} is bounded in absolute value by
\begin{align*}
&Cn^{-\Lambda'} \sum_{j_4,j_5 ,j_6=0}^{\Nt -1} \sup_{p,j}\left|\left< \hat{\varepsilon}_{\frac pn}, \partial_{\frac jn}\right>_\hten\right|^{M-3}\sup_{j,k}\left| \left<\partial_{\frac jn},\partial_{\frac kn}\right>_\hten\right|^A\left( \sup_p \sum_{j=0}^{\Nt -1} \left|\left< \hat{\varepsilon}_{\frac pn}, \partial_{\frac jn}\right>_\hten\right|\right)^3\\
&\qquad\; \le C\Nt^3 n^{-\Theta'},\end{align*}
where, using the fact that $R=M+2A$,
\[\Theta' = 2-(\frac 13 -\lambda)(R-12) + (M-3)\lambda + \frac A3 = 6 + (R-A)(2\lambda -\frac 13) - 15\lambda.\]
Observe that $\Theta' > 3$ whenever $R-A > 6$.  The case $R-A = 6$ occurs only when $R=12$, $A=6$, and $M=0$; so in this case we have an upper bound of 
\[Cn^{-2} \sum_{j_4, j_5, j_6=0}^{\Nt -1}  \sup_{j,k}\left| \left<\partial_{\frac jn},\partial_{\frac kn}\right>_\hten\right|^3\left(
\sup_k \sum_{j=0}^{\Nt -1} \left| \left<\partial_{\frac jn},\partial_{\frac kn}\right>_\hten\right|\right)^3\]
\[ \le C\Nt^3 n^{-2-\frac 63} \le Cn^{-1}.\]

\medskip
\noindent{\bf Acknowledgement}

The authors wish to thank two anonymous referees for a careful reading and valuable comments.

\end{document}